\documentclass[a4paper,11pt,twoside,reqno]{amsart} 
\DeclareSymbolFontAlphabet{\mathcal}{symbols}
\usepackage[cmtip,all]{xy}
\numberwithin{equation}{section}
\usepackage{amsmath}
\usepackage[psamsfonts]{amssymb}
\usepackage[psamsfonts]{eucal}
\usepackage{amsthm}
\usepackage[psamsfonts]{amssymb}
\usepackage[scaled]{helvet}
\usepackage{mathrsfs}
\usepackage{framed}
\usepackage{hyperref}
\hypersetup{
    colorlinks,
    citecolor=blue,
    filecolor=black,
    linkcolor=blue,
    urlcolor=black
}
\usepackage[a4paper]{geometry}
\usepackage{multicol}
\usepackage{enumerate}
\usepackage{graphicx}
\usepackage{tikz}
\usepackage{color}
\definecolor{trama}{gray}{.875}
\usepackage{lmodern}
\usepackage{textcomp}% symboles supplémentaires
\usepackage{enumerate}
\usepackage{wrapfig}
\title%[Variations on  intersection Poincar\'e  duality]
{Variations on Poincar\'e duality for intersection homology}
\date{\today}

\author{Martintxo Saralegi-Aranguren}
\address{Laboratoire de Math{\'e}matiques de Lens\\  
      EA 2462 \\
      Universit\'e d'Artois\\
         SP18, rue Jean Souvraz\\
          62307 Lens Cedex\\
         France}
\email{martin.saraleguiaranguren@univ-artois.fr}

\author{Daniel Tanr\'e}
\address{D\'epartement de Math{\'e}matiques\\
         Facult\'e des Sciences et Technologies, UMR 8524 \\
         Universit\'e de Lille\\
         59655 Villeneuve d'Ascq Cedex\\
         France}
\email{Daniel.Tanre@univ-lille.fr}

\thanks{
The second author was  supported by the MINECO and FEDER research project MTM2016-78647-P. 
and the ANR-11-LABX-0007-01  ``CEMPI''}



\makeatletter
\renewcommand\l@subsection{\@tocline{2}{0pt}{2pc}{5pc}{}}
\renewcommand\l@subsubsection{\@tocline{3}{0pt}{4pc}{10pc}{}}
\makeatother
\theoremstyle{plain}
\newtheorem{theorem}{Theorem}

\newtheorem{proposition}{Proposition}[section]
\newtheorem{theoremb}[proposition]{Theorem}
\newtheorem{lemma}[proposition]{Lemma}
\newtheorem{corollary}[proposition]{Corollary}

\theoremstyle{definition}
\newtheorem{definition}[proposition]{Definition}
\newtheorem{example}[proposition]{Example}

\theoremstyle{remark}
\newtheorem{remark}[proposition]{Remark}

\newcommand{\secref}[1]{Section~\ref{#1}}

\newcommand{\thmref}[1]{Theorem~\ref{#1}}
\newcommand{\propref}[1]{Proposition~\ref{#1}}
\newcommand{\lemref}[1]{Lemma~\ref{#1}}
\newcommand{\corref}[1]{Corollary~\ref{#1}}

\newcommand{\remref}[1]{Remark~\ref{#1}}
\newcommand{\examref}[1]{Example~\ref{#1}}

\newcommand{\defref}[1]{Definition~\ref{#1}}

\def\R{{\mathbb R}}

\def\ov{\overline}

\def\cP{{\mathcal P}}

\def\cS{{\mathcal S}}
\def\cU{{\mathcal U}}
\def\cV{{\mathcal V}}
\def\cW{{\mathcal W}}

\def\crC{{\mathscr C}}
\def\crD{{\mathscr D}}

\def\crF{{\mathscr F}}
\def\crG{{\mathscr G}}
\def\crH{{\mathscr{H}}}

\def\crK{{\mathscr K}}

\def\crN{{\mathscr N}}

\def\crT{{\mathscr T}}
\def\crR{{\mathscr R}}

\def\1{{\mathbf 1}}
\def\gd{{\mathfrak{d}}}
\def\gD{{\mathfrak{D}}}
\def\gC{{\mathfrak{C}}}
\def\gH{{\mathfrak{H}}}

\def\tc{{\mathtt c}}

\def\tu{{\mathtt u}}
\def\tv{{\mathtt v}}

\def\C{\mathbb{C}}

\def\N{\mathbb{N}}
\def\Q{{\mathbb{Q}}}
\def\R{\mathbb{R}}
\def\Z{\mathbb{Z}}

\def\ker{{\rm Ker\,}}
\def\coker{{\rm Coker\,}}

\def\Hom{{\rm Hom}}
\def\Ext{{\rm Ext}}

\def\depth{{\rm depth}}
\def\id{{\rm id}}

\def\codim{{\rm codim\,}}

\def\tN{{\widetilde{N}}}

\def\rc{{\mathring{\tc}}}

\def\Tors{{\mathrm T}}
\def\Free{{\mathrm F}}
\def\pN{{\tN_{D\ov{p}/\ov{p}}}}
\def\pNc{{\tN_{D\ov{p}/\ov{p},c}}}
\def\pC{{\gC^*_{\ov{p}/D\ov{p},c}}}
\def\pH{{\mathscr{H}}_{D\ov{p}/\ov{p}}}

 \newcommand{\menos}{\backslash} 
 \def\Ker{{\rm Ker\,}}
\def\Coker{{\rm Coker\,}}

\def\Hom{{\rm Hom}}
\def\Ext{{\rm Ext}}
\begin{document}

\begin{abstract} 
Intersection homology with coefficients in a field restores Poincar\'e duality 
for some spaces with singularities, as stratified pseudomanifolds. But, with coefficients in a ring, 
the behaviours of manifolds and stratified pseudomanifolds are different. 
This work is an overview, with proofs and explicit examples, of  various possible  situations with their properties.

We first  set up  a duality, defined from a cap product, between two intersection cohomologies:
the first one arises from a linear dual and the second one from a simplicial blow up.
Moreover, from this property, Poincar\'e duality in intersection homology looks like 
the Poincar\'e-Lefschetz duality of a manifold with boundary.
Besides that, an investigation of the  coincidence of the two previous cohomologies  reveals that 
the only obstruction to the existence of a Poincar\'e duality  is the homology of a well defined complex. 
This recovers the case of  the peripheral sheaf introduced by Goresky and Siegel 
for compact PL-pseudomanifolds.
We also list a series of explicit computations of  peripheral intersection cohomology. In particular, we observe
that 
Poincar\'e duality can exist in the presence of torsion in the ``critical degree'' of the intersection homology of the links
of a stratified pseudomanifold.
\end{abstract}

\maketitle

\tableofcontents

\newpage
%%%%%%%%%%%%%%%%%%%%%%%%%
\section*{Introduction}

\begin{quote}
In this introduction, for sake of simplicity,  we  restrict the coefficients to $\Z$ and~$\Q$.
We consider also Goresky and MacPherson perversities, depending only on the codimensions of strata. 
A more general situation is handled in the text
and specified in the various statements. 
Recollections of definitions and main properties can be found in \secref{sec:back}.
\end{quote}

Let $M$ be a compact, $n$-dimensional, oriented manifold.
The famous 
Poincar\'e duality gives a non-singular pairing
$$H_k(M;\Q)\otimes H_{n-k}(M;\Q)\to \Q,$$
defined by the intersection product.
This feature has been extended to the existence of singularities by M. Goresky and R. MacPherson.
In \cite{GM1}, they introduce the \emph{intersection homology} associated to a perversity $\ov{p}$
and prove the existence of a non-singular pairing in intersection homology,
$$H_k^{\ov{p}}(X;\Q)\otimes H_{n-k}^{D\ov{p}}(X;\Q)\to \Q,$$
when $\ov{p}$ and $D\ov{p}$ are complementary perversities and $X$ is a 
compact, oriented, $n$-dimensional PL-pseudomanifold.
If we replace the field of rational numbers by the ring of integers, the situation becomes more complicated. 
In the case of a compact oriented manifold, we still have  non-singular pairings, 
\begin{equation}\label{equa:freemanifold}
\Free\,H_k(M;\Z)\otimes \Free\,H_{n-k}(M;\Z)\to \Z,
\end{equation}
between the torsion free  parts of homology groups, and
\begin{equation}\label{equa:torsionmanifold}
\Tors H_k(M;\Z)\otimes \Tors H_{n-k-1}(M;\Z)\to \Q/\Z
\end{equation}
between the torsion parts.
In contrast, these two properties can disappear in  intersection homology  as it has been 
discovered and studied by M.~Goresky and P.~Siegel in \cite{GS}. 
In their work, they define a class of compact PL-pseudomanifolds 
called \emph{locally $\ov{p}$-torsion free} (see \defref{def:locallyfree}) 
for which there exist non-singular pairings in intersection homology,
\begin{equation}\label{equa:free}
\Free\, H_k^{\ov{p}}(X;\Z)\otimes \Free\, H_{n-k}^{D\ov{p}}(X;\Z)\to \Z
\end{equation}
and
\begin{equation}\label{equa:torsion}
\Tors \, H_k^{\ov{p}}(X;\Z)\otimes \Tors\, H_{n-k-1}^{D\ov{p}}(X;\Z)\to \Q/\Z.
\end{equation}
But there are examples of PL-pseudomanifolds for which the previous pairings are singular,
as for example the Thom space associated to the tangent space of the 2-sphere for (\ref{equa:free}) and the suspension of the real projective space $\R P^3$ for  (\ref{equa:torsion}),
see Examples \ref{exam:thoms2} and \ref{exam:susprp3cie}.

Let us come back to recollections on Poincar\'e duality. For an  oriented, $n$-dimensional manifold,
$M$,  the cap product with the fundamental class is an isomorphism,
\begin{equation}\label{equa:capdual}
-\frown [M]\colon H^k_{c}(M;\Z)\xrightarrow{\cong}H_{n-k}(M;\Z),
\end{equation}
between  cohomology with compact supports and homology. The existence of the 
non-singular pairings
(\ref{equa:freemanifold}) and (\ref{equa:torsionmanifold}) are then consequences
 of (\ref{equa:capdual}) and  the universal coefficient formula.
 In intersection homology, this method was investigated
by G. Friedman and J.E. McClure in \cite{FM} and taken over in \cite[Section 8.2]{FriedmanBook}.
As cohomology groups, the authors consider the homology of the dual 
$ C^*_{\ov{p}}(X;\Z)=\Hom( C_*^{\ov{p}}(X;\Z),\Z)$ of the complex of $\ov{p}$-intersection chains
and, under the same restriction as in Goresky and Siegel's paper, they prove the existence
of an isomorphism induced by a cap product with a fundamental class,
\begin{equation}\label{equa:capiso}
 H^k_{\ov{p},c}(X;\Z)\xrightarrow{\cong} H_{n-k}^{D\ov{p}}(X;\Z),
\end{equation}
for a locally $\ov{p}$-torsion free, oriented, paracompact, $n$-dimensional stratified pseudomanifold $X$. 
With this restriction, the pairings (\ref{equa:free}) and (\ref{equa:torsion})
are then deduced from (\ref{equa:capiso}) and a formula of universal coefficients, as in the case of a manifold.

In \cite{CST2}, we take over the approach  (\ref{equa:capdual}) 
using blown-up cochains with compact supports, 
$\tN^*_{\ov{\bullet},c}(-)$,
that we have introduced and studied in  previous papers 
\cite{CST3}, \cite{CST6}, \cite{CST7},  \cite{CST5}, \cite{CST4}, \cite{CST1}  
(also called Thom-Whitney cochains in some of these  works).
One of their features  is the existence of cup and cap products
(see \cite{CST4} or \secref{sec:back})
for any ring of coefficients and
without any restriction on the stratified pseudomanifold.
Indeed we prove in \cite[Theorem B]{CST2} that, 
for any  oriented, paracompact, $n$-dimensional stratified pseudomanifold, $X$, and any perversity $\ov{p}$, 
the cap product with a fundamental class is an isomorphism,
\begin{equation}\label{equa:bup}
-\frown [X]\colon \crH^k_{\ov{p},c}(X;\Z)\xrightarrow{\cong}  H_{n-k}^{\ov{p}}(X;\Z),
\end{equation}
between the blown-up cohomology 
with compact supports $\crH^k_{\ov{p},c}(-)$ and the intersection homology.
The blown-up cohomology is not defined from the dual complex of intersection chains 
but proceeds from a simplicial blow up process recalled in \secref{sec:back}.
Thus, there is no universal coefficients formula between 
$\crH^*_{\ov{p}}(-)$ and $ H_*^{\ov{p}}(-)$
and we cannot deduce from (\ref{equa:bup})  a non-singular bilinear form
as in the classical case of a manifold.
In \cite[Theorem C]{CST2}, for a compact oriented stratified 
pseudomanifold $X$, we prove the non-degeneracy of the bilinear form 
\begin{equation}\label{equa:freepairing} 
\Phi_{\ov{p}} \colon  \Free\crH^k_{\ov{p}}(X;\Z)\otimes \Free\crH^{n-k}_{D\ov{p}}(X;\Z)\to \Z
\end{equation}
built from the cup product.
(There are examples of stratified pseudomanifolds where this bilinear form is singular,
see \cite[Example 4.10]{CST2} or \examref{exam:thoms2}.)
In contrast, there are examples (see \examref{exam:susprp3cie}) of the degeneracy of the associated bilinear form 
\begin{equation}\label{equa:torpairing}
L_{\ov{p}} \colon 
\Tors \crH^k_{\ov{p}}(X;\Z)\otimes \Tors \crH^{n+1-k}_{D\ov{p}}(X;\Z)\to \Q/\Z. 
\end{equation}
The existence of such examples is  not surprising: as the blown-up cohomology is isomorphic through (\ref{equa:bup})
to the intersection homology, the defect of duality detected by Goresky and Siegel is
also present in (\ref{equa:freepairing}) and (\ref{equa:torpairing}).
In sum, we have two intersection  cohomologies, $ H_{\ov{\bullet}}^{\ast}(-)$ and $\crH_{\ov{\bullet}}^{\ast}(-)$:
the first one has a universal coefficient formula and the second one 
satisfies the isomorphism (\ref{equa:bup}) through a cap product with a fundamental class.
But, as the quoted examples show, neither satisfies a Poincar\'e duality with cup products 
and coefficients in $\Z$, in all generality.
(However, the blown-up cohomology satisfies (\ref{equa:bup}) over $\Z$
without restriction on the torsion of links.)

This work is also concerned with not necessarily compact stratified pseudomanifolds and, for having
a complete record, let us also mention the existence of an isomorphism,
\begin{equation}\label{equa:bupBM}
-\frown [X]\colon \crH^k_{\ov{p}}(X;\Z)\xrightarrow{\cong}  H^{\infty,\ov{p}}_{n-k}(X;\Z),
\end{equation}
between the blown-up intersection cohomology and the Borel-Moore intersection homology,
(see \cite{ST1} or \cite{CST5} in the PL case) for any
paracompact, separable and oriented  stratified pseudomanifold of dimension $n$.

After this not so brief ``state of the art'', we present the results of this work.
\emph{The starting point is the existence of a duality between the two intersection cohomologies,}
developed in \secref{sec:capdual}.
To express it, we use the injective resolution,
$I_{\Z}^*\colon \Q \to \Q/\Z$,
 and the Verdier dual, $DA^*$,   
 defined as the Hom functor of a cochain complex $A^*$ with value in $I_{\Z}^*$,
 see (\ref{equa:dualcochain}).
  
%\margen{Find a solution for the notation $\ast$.}
\begin{theorem}{\rm{[\thmref{thm:duality}]}}\label{thm:poinclef}
Let $X$ be a  
paracompact, separable and oriented  stratified pseudomanifold of dimension $n$
and $\ov{p}$ be a perversity.
Then, there exist two quasi-isomorphisms, defined from the cap product
with a cycle representing the fundamental class $[X]\in H_n^{\infty,\ov{0}}(X;\Z)$,
$$
\crC_{\ov{p}}
\colon
 C^{\star}_{\ov{p}}(X;\Z)\to (D\tN^{*}_{\ov{p},c}(X;\Z))_{n-\star}
\text{ and }
\crN_{\ov{p}}
\colon
\tN^{\star}_{\ov{p}}(X;\Z)\to (D C^{*}_{\ov{p},c}(X;\Z)_{n-\star}.
$$
\end{theorem}

As a consequence, in the compact case, we deduce two non-singular pairings between the two intersection cohomologies,
\begin{equation}\label{equa:dualbupH}
\Free  H^k_{\ov{p}}(X;\Z)\otimes \Free \crH^{n-k}_{\ov{p}}(X;\Z)\to \Z
\text{ and }
\Tors  H^k_{\ov{p}}(X;\Z)\otimes \Tors \crH^{n-k+1}_{\ov{p}}(X;\Z)\to \Q/\Z.
\end{equation}

\emph{In a second step, we are looking for a quasi-isomorphism between
$\tN^*_{\ov{p}}(X;\Z)$ and $D\tN^*_{D\ov{p},c}(X;\Z)$.} This can be deduced from \thmref{thm:poinclef}
and the existence of a quasi-isomorphism between 
$\tN^*_{\ov{p}}(X;\Z)$ and  $ C^*_{D\ov{p}}(X;\Z)$.
For investigating that, we use the existence (see \propref{prop:chipq}) of a cochain map,
$\chi_{\ov{p}}\colon \tN^*_{\ov{p}}(X;\Z)\to  C^*_{D\ov{p}}(X;\Z)$,
and its version with compact supports,
$\chi_{\ov{p},c}\colon \tN^*_{\ov{p},c}(X;\Z)\to  C^*_{D\ov{p},c}(X;\Z)$.
So, by setting $\crD_{\ov{p}}= \crC_{D\ov{p}} \circ \chi_{\ov{p}}$, we get a cochain map,
\begin{equation}\label{equa:dualblownup}
\crD_{\ov{p}}\colon \tN^*_{\ov{p}}(X;\Z)\to D\tN^*_{D\ov{p},c}(X;\Z),
\end{equation}
which is a quasi-isomorphism if, and only if, the map $\chi_{\ov{p}}$ is a quasi-isomorphism.
Hence, the homotopy cofiber  of $\chi_{\ov{p}}$ in the category of cochain complexes plays a fundamental role
in Poincar\'e duality. 
We study it in \secref{sec:peripheralGS}. We call it
 the \emph{peripheral complex} and denote it and its homology by 
$R^*_{\ov{p}}$ and  $\crR^*_{\ov{p}}$, respectively.
(A brief analysis  shows that it
corresponds effectively to the global sections of the peripheral sheaf of \cite{GS}, in
the PL compact case.) 
This complex, which personifies the non-duality, owns itself a duality in the compact case.
To write it in our framework, we introduce the compact supports analogues,
$R^*_{\ov{p},c}$ and  $\crR^*_{\ov{p},c}$,
of $R^*_{\ov{p}}$ and  $\crR^*_{\ov{p}}$.

\begin{theorem}{\rm{[\thmref{thm:dualityperiph}]}}\label{thm:periph}
Let $X$ be a  
paracompact, separable and oriented stratified pseudomanifold of dimension $n$
and  $\ov{p}$ be a perversity.
Then, there exists a quasi-isomorphism,
$$R^{\star}_{\ov{p}}(X;\Z)\to (DR^{*}_{D\ov{p},c}(X;\Z))_{n-\star-1}.$$
\end{theorem}
 
We also describe some properties of this complex, established 
in \cite{GS}  for PL compact stratified pseudomanifolds. 
For instance, as $\chi_{\ov{p}}$ induces a quasi-isomorphism when the ring of coefficients is a field,
the homology $\crR^*_{\ov{p}}(X;\Z)$ is entirely torsion.
 As the nullity of $\crR^*_{\ov{p}}(X;\Z)$  is a sufficient and necessary condition for having the
 quasi-isomorphism $\crD_{\ov{p}}$, we may enquire what means the
 ``locally $\ov{p}$-torsion free''  requirement appearing in \cite{GS} and \cite{FM}.
 In \propref{prop:localfreeacyclic}, we  show that it is equivalent to the nullity
 of the peripheral cohomology $\crR^*_{\ov{p}}(U;\Z)$ for \emph{any open subset of $X$.} 
 \examref{exam:periphnofree} shows that this last property
 is not necessary for getting the quasi-isomorphism $\crD_{\ov{p}}$.

Suppose $\ov{p}\leq D\ov{p}$. We denote
 by $\pN(X;R)$ the homotopy cofiber of the inclusion of cochain complexes,  
$\tN_{\ov{p}}^*(X;R)\to \tN_{D\ov{p}}^*(X;R)$
and by $\pNc(X;R)$ the compact support version of it.
 In \cite[Lemma 3.7]{MR3028755}, G. Friedman and E. Hunsicker
 prove that the homology analogue of this relative complex owns a self-duality  for  compact PL-pseudomanifolds
 and intersection homology with rational coefficients.
 In \secref{sec:relativeDp}, we extend this result to 
 paracompact, separable and oriented  stratified pseudomanifolds of dimension $n$, $X$.
Denote by $C^*_{\ov{p}/D\ov{p},c}(X;\Z)$ the cofiber of the inclusion
$ C_{D\ov{p},c}^*(X;\Z)\to  C^*_{\ov{p},c}(X;\Z)$.
In \propref{prop:petitpasDpp}, we get a quasi-isomorphism, similar to the one of
\thmref{thm:duality},
$$
\tN_{D\ov{p}/\ov{p}}^{\star}(X;\Z)
\to (DC^*_{\ov{p}/D\ov{p},c})_{n-\star}.
$$
Next, if $\chi_{\ov{p}}$ and  $\chi_{\ov{p},c}$ are quasi-isomorphisms,  
 we prove the existence of a quasi-isomorphism,
 $$\pN^{\star}(X;\Z)\to (D \pNc^{*}(X;\Z))_{n-\star-1},$$
 which gives back the self-duality of \cite{MR3028755}, see \corref{cor:petitpasDpp}. 
  
  \medskip
 In \secref{sec:periphmas}, we study some components of the peripheral cohomology
 for compact oriented stratified pseudomanifolds. 
The   pairings deduced from (\ref{equa:dualbupH}) are investigated separately
 for the existence of  non-singular pairings in
the torsion or in the torsion free parts. 
In \secref{sec:examples}, examples of the different possibilities are described.
%They are described in \secref{sec:examples}.
In particular, \examref{exam:periphnofree}  is a 
not locally $\ov{p}$-torsion free stratified pseudomanifold with 
Poincar\'e duality over $\Z$.
Finally, let us emphasize that most of the duality results in Sections 
\ref{sec:capdual}, \ref{sec:peripheralGS}, \ref{sec:relativeDp}  do not require an hypothesis of finitely
generated homology.

\medskip
\paragraph{\bf Notations and conventions}
 In  this work, 
homology and cohomology are considered with coefficients in a principal ideal domain, $R$, 
or in its field of fractions $QR$ and, if there is no ambiguity, 
we do not mention the coefficient explicitly in the proofs.
For any $R$-module, $A$,
we denote by $\Tors A$ the \emph{$R$-torsion submodule} of $A$ and by 
$\Free A=A/\Tors A$ the $R$-torsion free quotient of $A$. Recall that a pairing 
$A\otimes B\to R$ is \emph{non-degenerate} if the two adjunction maps, 
$A\to \Hom (B,R)$ and $B\to \Hom(A,R)$,
are injective. The pairing is \emph{non-singular} if they are both isomorphisms.

For any topological space $X$, we denote by $\tc X=X\times [0,1]/X\times\{0\}$
the cone on $X$
and by $\rc X=X\times [0,1[/X\times\{0\}$ the open cone on $X$.
Elements of the cones are denoted $[x,t]$ and the apex is $\tv=[-,0]$.

In the previous introduction, $H_{i}^{\ov{p}}(-)$ denotes the intersection homology of
\cite{GM1} or \cite{MR642001}. It can be obtained from the chain complex of filtered simplices
of \defref{def:filteredsimplex}, see \cite[Proposition A.29]{CST1}.
The perversities used in this work are completely general: they are defined on the set of strata
and do not only depend on the codimension. Moreover, we lift any restriction on the values taken by
a perversity. An issue of that freedom is that an allowable simplex in the sense of 
\cite{GM1} or \cite{MR642001}
may have a support totally included in the singular subset. 
This has bad consequences, as the breakdown of Poincar\'e duality. To overcome this failure,
we use a complex built from filtered simplices that are not totally included in the singular set,
see \remref{rem:nongm}.
As it differs from the complex of
\cite{GM1} or \cite{MR642001},
we denote it by $\gC_{*}^{\ov{p}}(-)$ and its homology by $\gH_{*}^{\ov{p}}(-)$.
We emphasize that for the original perversities of the loc. cit. references, we have
$C_{*}^{\ov{p}}(-)= \gC_{*}^{\ov{p}}(-)$
and
$H_{*}^{\ov{p}}(-)= \gH_{*}^{\ov{p}}(-)$.
Simply, our approach allows an extension of the original historical definition that leaves it unchanged.
Thus we call it intersection homology without ambiguity.

The dual complex $\gC^*_{\ov{p}}(-)=\Hom(\gC_{*}^{\ov{p}}(-),R)$  gives birth to a cohomology $\gH^*_{\ov{p}}(-)$. As explained
before in the introduction, this cohomology does not
satisfy a Poincar\'e duality, through a cap product, with intersection homology for any coefficients. For having this property,
we use a cohomology constructed from a simplicial blow up.
For a clear distinction with the previous cohomology obtained with a linear dual, we 
 denote $\crH^*_{\ov{p}}(-)$ the blown-up cohomology and $\tN^*_{\ov{p}}(-)$ its corresponding cochain complex .

%%%%%%%%%%%%%%%%%%%%%%
\section{Background}\label{sec:back}

\begin{quote}
We recall the basics we need, sending the reader to   \cite{CST1}, \cite{CST4}, \cite{FriedmanBook} or \cite{GM1}, 
for more details.
\end{quote}

%%%%%%%%%%%%%%%
\subsection{\tt Pseudomanifolds}
First come the geometrical objects, the stratified pseudomanifolds. In this work, we authorize 
them to have strata of codimension~1.

\begin{definition}\label{def:pseudomanifold}
A \emph{topological stratified pseudomanifold of dimension $n$} (or a stratified pseudomanifold) is 
a Hausdorff  space
together with a filtration by closed subsets,
$$
X_{-1}=\emptyset\subseteq X_0 \subseteq X_1 \subseteq \dots \subseteq X_{n-2} \subseteq X_{n-1} \subsetneqq X_n =X,
$$
such that, for each $i\in\{0,\dots,n\}$, 
$X_i\backslash X_{i-1}$ is a topological manifold of dimension $i$ or the empty set. 
The subspace $X_{n-1}$ is called \emph{the singular set} and each 
point $x \in X_i \backslash X_{i-1}$ with $i\neq n$ admits
\begin{enumerate}[(i)]
\item an open neighborhood $V$ of $x$ in $X$, endowed with the induced filtration,
\item an open neighborhood $U$ of $x$ in  $X_i\backslash X_{i-1}$, 
\item a  compact stratified pseudomanifold $L$  of dimension  $n-i-1$, whose cone $\rc L$ is endowed with the conic filtration, 
$(\rc L)_{i}=\rc L_{i-1}$,
\item a homeomorphism, $\varphi \colon U \times \rc L\to V$, 
such that
\begin{enumerate}[(a)]
\item $\varphi(u,\tv)=u$, for any $u\in U$, where $\tv$  is the apex of $\rc L$,
\item $\varphi(U\times \rc L_{j})=V\cap X_{i+j+1}$, for any $j\in \{0,\dots,n-i-1\}$.
\end{enumerate}
\end{enumerate}
A topological stratified pseudomanifold of dimension~0 is a discrete set of points.

The stratified pseudomanifold $L$ is called the \emph{link} of $x$. 
The connected components $S$ of $X_{i}\backslash X_{i-1}$ are the  \emph{strata} of $X$ of dimension~$i$.
The strata of dimension $n$  are said to be \emph{regular} and we denote by
$\cS_X$ (or $\cS$ if there is no ambiguity) the set of non-empty strata. 
We have proven in \cite[Proposition A.22]{CST1} that  $S\preceq S'$ if, and only if,   $S\subset \overline{S'}$, 
defines an order relation. We also denote $S\prec S'$ if $S\preceq S'$ and $S\neq S'$.
\end{definition}

\defref{def:pseudomanifold} of stratified pseudomanifold is slightly more general 
than the one in \cite{GM2} where it is supposed $X_{n-1}=X_{n-2}$.
In this work, we are concerned with Poincar\'e duality and general perversities, 
for which the previous restriction is not necessary.
On the other hand, the hypothesis $X_{n}\neq X_{n-1}$ implies that the links 
of the singular strata are always non-empty sets, therefore
$X_{n}\menos X_{n-1}$ is dense in $X$.
This infers a ``good'' notion of dimension on $X$ which is the relevant point in  \cite[Page 82]{GM2}
and motivates us for keeping the appellation of pseudomanifold in this case.

  \begin{example}
  Among stratified pseudomanifolds, let us quote the manifolds,  the open subsets of a stratified pseudomanifold 
  (with the induced structure), the cones on  compact manifolds with the singular set reduced to the apex, 
  the Thom spaces filtered by the compactification point.
  As relevant examples of spaces admitting a structure of stratified pseudomanifolds, 
  we may also take over the list of \cite{GM2}:
  complex algebraic varieties, complex analytic varieties, real analytic varieties, Whitney stratified sets,
  Thom-Mather stratified spaces.
  For instance, the following picture represents the real part of the 
  hypersurface of $\C^3$ called
  Whitney cusp, 
   with its stratification,

 \centerline{ \includegraphics[scale=0.7]{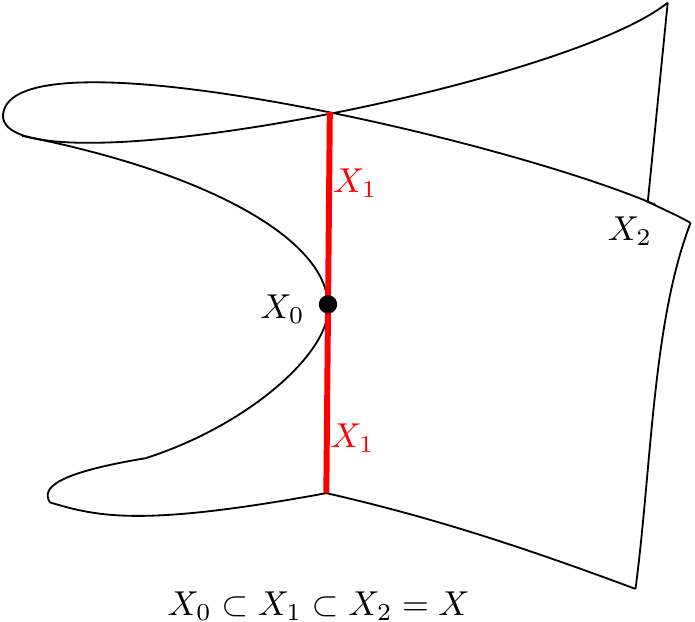}} 
  \end{example}

 \begin{definition}\label{depth}
 The \emph{depth} of a topological stratified pseudomanifold $X$ is the largest integer $\ell$ 
 for which there exists a chain of strata,
 $
 S_0 \prec S_1 \prec \cdots \prec S_\ell.
 $
 It is denoted by $\depth \,X$.
 \end{definition}
 In particular, $\depth \,X=0$ if, and only if, all the strata of $X$ are regular.

%%%%%%%%%%%%
\subsection{\tt Perversity}
The second concept in intersection homology is that of perversity. We consider the perversities of \cite{MacPherson90}
defined on each stratum. They are already used in \cite{MR1245833}, \cite{MR2210257}, \cite{FR2}, 
  \cite{MR2796412},
\cite{FM}.

\begin{definition}\label{def:perversite} 
A  \emph{perversity on a stratified pseudomanifold,} $X$, is a map,
$\ov{p}\colon \cS_X\to \Z$, defined on the set of strata of $X$ and taking the value~0 on the regular strata.
The pair $(X,\ov{p})$ is called a \emph{perverse pseudomanilfold.}
If $\ov{p}$ and $\ov{q}$ are two perversities on $X$, we set  $\ov{p}\leq \ov{q}$ if we have
$\ov{p}(S)\leq\ov{q}(S)$, for all $S\in \cS_{X}$.
\end{definition}

Among perversities, there are those considered in \cite{GM1} and whose values depend only on the codimension of
the strata. 

\begin{definition}\label{def:GMperversite}
A \emph{{\rm GM}-perversity} is a map $\ov{p}\colon \N\to\Z$ such that
$\ov{p}(0)=\ov{p}(1)=\ov{p}(2)=0$ and
$\ov{p}(i)\leq\ov{p}(i+1)\leq\ov{p}(i)+1$, for all $i\geq 2$.
As particular case, we have the null perversity~$\ov{0}$ constant with value~0 and the \emph{top perversity} defined by
$\ov{t}(i)=i-2$ if $i\geq 2$. 
For any perversity, $\ov{p}$, the perversity $D\ov{p}:=\ov{t}-\ov{p}$ is called the
\emph{complementary perversity} of $\ov{p}$.
A GM-perversity induces a perversity on $X$ by $\ov{p}(S)=\ov{p}(\codim S)$.
\end{definition}

\begin{example}
Let us mention the lower-middle and the upper-middle perversities, respectively defined on the singular strata
by
 $$\ov{m}(S)=\left\lfloor \frac{(\codim S)-2}{2} \right\rfloor
 \text{ and }
\ov{n}(S)=D\ov{m}(S)=\left\lceil \frac{(\codim S)-2}{2} \right\rceil ,
 $$
 which play an important role in intersection homology.  
 They coincide for Witt spaces (\cite[Definition 11.1]{GS})
  and, for them, a non-singular  pairing exists in intersection homology with rational coefficients, 
  see \cite{GS}. 
  For instance, this is the case for the Thom space of the tangent bundle of the 2-sphere and there is a pairing induced by 
  the cup product,
  $\crH_{\ov{m}}^k(X;\Q) \otimes \crH_{\ov{m}}^{n-k}(X;\Q)
   \to
   \Q$.
  Its behaviour with integer coefficients is analyzed in \cite[Example 4.10]{CST2}.
 \end{example}
%%%%%%%
\subsection{\tt Intersection Homology}\label{subsec:homologieintersection}

We specify the chain complex used for the determination of intersection homology of a stratified pseudomanifold $X$
equipped with a perversity $\ov{p}$,
cf. \cite{CST3}.

\begin{definition}\label{def:filteredsimplex}
A  \emph{filtered simplex} is a continuous map $\sigma\colon\Delta\to X$, 
from a Euclidean simplex endowed with a decomposition
$\Delta=\Delta_{0}\ast\dots\ast\Delta_{n}$,
 called \emph{$\sigma$-decomposition of $\Delta$},
such that
$
\sigma^{-1}X_{i} =\Delta_{0}\ast\dots\ast\Delta_{i}
$,
for all~$i \in \{0, \dots, n\}$, where $\ast$ denotes the join. 
The sets  $\Delta_{i}$  may be empty, with the convention $\emptyset * Y=Y$, for any space $Y$. 
The simplex $\sigma$ is  \emph{regular} if $\Delta_{n}\neq\emptyset$. A chain is  \emph{regular} if it is a linear combination of regular simplices. 
\end{definition}

Given a Euclidean regular simplex $\Delta = \Delta_0 * \dots *\Delta_n$, we consider   
as ``boundary'' of $\Delta$ the regular part
$\gd\Delta$  of the chain $\partial \Delta$.  
That is
$\gd \Delta =\partial (\Delta_0 * \dots * \Delta_{n-1})* \Delta_n$, if $|\Delta_n| = 0 $, or $\gd \Delta = \partial \Delta$,
if $|\Delta_n|\geq 1$.
For any regular simplex $\sigma \colon\Delta \to X$, we set  $\gd \sigma=\sigma_* \circ \gd$
 and denote by $\gC_{*}(X;R)$ 
 the  complex of  linear combinations of regular  simplices (called finite chains)
with the differential $\gd$.

 \begin{definition}\label{def:lessimplexes}
 The \emph{perverse degree of} a filtered simplex $\sigma\colon\Delta=\Delta_{0}\ast \dots\ast\Delta_{n} \to X$
  is the $(n+1)$-uple,
$\|\sigma\|=(\|\sigma\|_0,\dots,\|\sigma\|_n)$,  
with 
 $\|\sigma\|_{i}=
\dim (\Delta_{0}\ast\dots\ast\Delta_{n-i})$
 and the  convention $\dim \emptyset=-\infty$.
The \emph{perverse degree of $\sigma$ along a stratum $S$} is defined by
  $$\|\sigma\|_{S}=\left\{
 \begin{array}{cl}
 -\infty,&\text{if } S\cap \sigma(\Delta)=\emptyset,\\
 \|\sigma\|_{\codim S},&\text{otherwise.}
  \end{array}\right.$$
 A \emph{regular simplex is $\ov{p}$-allowable} if
  \begin{equation}\label{equa:admissible}
  \|\sigma\|_{S}\leq \dim \Delta-\codim S+\ov{p}(S)
  =\dim\Delta-D\ov{p}(S)-2,
  \end{equation}
   for each stratum $S$ of $X$. A chain $\xi$ is 
   \emph{$\ov{p}$-allowable} if it is a linear combination of $\ov{p}$-allowable simplices,
   and of  \emph{$\ov{p}$-intersection} if  $\xi$ and its boundary $\gd\xi$ are $\ov{p}$-allowable.
   Let
$\gC_{*}^{\ov{p}}(X;R)$ be
the complex  of  $\ov{p}$-intersection chains and  
 $\gH^{\ov{p}}_*(X;R)$ its homology, called \emph{$\ov{p}$-intersection homology}.
 \end{definition}
 
 \begin{remark}\label{rem:nongm}
  This homology is called tame intersection homology in \cite{CST3}
  and non-GM intersection homology in  \cite{FriedmanBook}, see \cite[Theorem B]{CST3}.
  It coincides with the
 intersection homology for the original perversities of \cite{GM1}, see \cite[Remark 3.9]{CST3}.
 \end{remark}

We introduce also  
the complex  $\gC^{\infty,\ov{p}}_{*}(X;R)$ of  locally finite   chains of $\ov{p}$-intersection
with the differential $\gd$.
If  $X$ is locally compact, metrizable and separable, this complex is isomorphic 
(see \cite[Proposition 3.4]{CST5}) to the inverse limit,
 $$\gC^{\infty,\ov{p}}_{*}(X;R) \cong \varprojlim_{K\subset X}\gC^{\ov{p}}_{*}(X,X\menos K;R),
 $$
where $K$ runs over the compact subsets of $X$.
Its homology, 
$\gH^{\infty,\ov{p}}_{*}(X;R)$,
is called \emph{the locally finite (or Borel-Moore) $\ov p$-intersection homology.}

%%%%%%%%%%%%
\subsection{\tt Blown-up cohomology}
Let   
$N^*(\Delta)$ be the  simplicial  cochain complex
of  a Euclidean simplex $\Delta$, with coefficients in $R$. 
Given a face $F$  of $\Delta$, we write $\1_{F}$ for the element of $N^*(\Delta)$ 
taking the value 1 on $F$ and 0 otherwise. 
We denote also by $(F,0)$ the same face viewed as face of the cone $\tc\Delta=[\tv]\ast \Delta $ and by $(F,1)$ 
the face $\tc F$ of $\tc \Delta$. 
The apex is denoted $(\emptyset,1)=\tc \emptyset =[\tv]$.
Cochains on the cone $\tc\Delta$ are denoted $\1_{(F,\varepsilon)}$ with $\varepsilon =0$ or $1$.
If $\Delta=\Delta_{0}\ast \dots\ast\Delta_{n} $, let us set
$$\tN^*(\Delta)=N^*(\tc\Delta_0)\otimes\dots\otimes N^*(\tc\Delta_{n-1})\otimes N^*(\Delta_n).$$
A basis of $\tN^*(\Delta)$ is formed of the elements 
$\1_{(F,\varepsilon)}=\1_{(F_{0},\varepsilon_{0})}\otimes\dots\otimes \1_{(F_{n-1},\varepsilon_{n-1})}\otimes \1_{F_{n}}$,
 where 
$\varepsilon_{i}\in\{0,1\}$ and
$F_{i}$ is a face of $\Delta_{i}$ for $i\in\{0,\dots,n\}$ or the empty set with $\varepsilon_{i}=1$ if $i<n$.
We set
$|\1_{(F,\varepsilon)}|_{>s}=\sum_{i>s}(\dim F_{i}+\varepsilon_{i})$.

%%%%%%%%%%%%%%%%%%%%%%%
%%%%%%%%%%%%%%%%
\begin{definition}\label{def:degrepervers}
Let $\ell\in \{1,\ldots,n\}$.
The  \emph{$\ell$-perverse degree} of 
$\1_{(F,\varepsilon)}\in \tN^*(\Delta)$ is
$$
\|\1_{(F,\varepsilon)}\|_{\ell}=\left\{
\begin{array}{ccl}
-\infty&\text{if}
&
\varepsilon_{n-\ell}=1,\\
|\1_{(F,\varepsilon)}|_{> n-\ell}
&\text{if}&
\varepsilon_{n-\ell}=0.
\end{array}\right.$$
For a cochain $\omega = \sum_b\lambda_b \ \1_{(F,\varepsilon)_{b} }\in\tN^*(\Delta)$ with 
$\lambda_{b}\neq 0$ for all $b$,
the \emph{$\ell$-perverse degree} is
$$\|\omega \|_{\ell}=\max_{b}\|\1_{(F,\varepsilon)_{b}}\|_{\ell}.$$
By convention, we set $\|0\|_{\ell}=-\infty$.
\end{definition}

Let $\sigma\colon \Delta=\Delta_0\ast\dots\ast\Delta_n\to X$ be a filtered simplex.
We set $\tN^*_{\sigma}=\tN^*(\Delta)$.
If $\delta_{\ell}\colon \Delta' 
\to\Delta$ 
 is  an inclusion of a face of codimension~1,
  we have 
$\partial_{\ell}\sigma=\sigma\circ\delta_{\ell}\colon 
\Delta'\to X$.
If $\Delta=\Delta_{0}\ast\dots\ast\Delta_{n}$ is filtered, the induced filtration on $\Delta'$ is denoted
$\Delta'=\Delta'_{0}\ast\dots\ast\Delta'_{n}$ and
$\partial_{\ell}\sigma$ is a filtered simplex.
The \emph{blown-up intersection complex} of $X$ is the cochain complex 
$\tN^*(X)$ composed of the elements $\omega$
associating to each regular filtered simplex
 $\sigma\colon \Delta_{0}\ast\dots\ast\Delta_{n}\to X$
an element
 $\omega_{\sigma}\in \tN^*_{\sigma}$  
such that $\delta_{\ell}^*(\omega_{\sigma})=\omega_{\partial_{\ell}\sigma}$,
for any face operator
 $\delta_{\ell}\colon\Delta'\to\Delta$
 with $\Delta'_{n}\neq\emptyset$. 
 The differential $d \omega$ is defined by
 $(d \omega)_{\sigma}=d(\omega_{\sigma})$.
 The \emph{perverse degree of $\omega$ along a singular stratum $S$} equals
 $$\|\omega\|_S=\sup\left\{
 \|\omega_{\sigma}\|_{\codim S}\mid \sigma\colon \Delta\to X \;
 \text{regular such that }
 \sigma(\Delta)\cap S\neq\emptyset
 \right\}.$$
 We denote $\|\omega\|$ the map which associates $\|\omega\|_S$
 to any singular stratum $S$ and~0 to any regular one.
 A \emph{cochain $\omega\in\tN^*(X;R)$ is $\ov{p}$-allowable} if $\| \omega\|\leq \ov{p}$ 
 and of \emph{$\ov{p}$-intersection} if $\omega$ and $d\omega$ are $\ov{p}$-allowable. 
 Let $\tN^*_{\ov{p}}(X;R)$ be the complex of $\ov{p}$-intersection cochains and  
 $\crH_{\ov{p}}^*({X};R)$ its homology, called
 \emph{blown-up $\ov{p}$-intersection cohomology}  of $X$.

Finally, we mention the existence of a version with compact supports,
$\tN^*_{\ov{p},c}(X;R)$ and $\crH^*_{\ov{p},c}(X;R)$,
whose properties have been established in  \cite{CST2}. 

\subsection{\tt Products}\label{subsec:products}
Let  $X$ be a stratified pseudomanifold equipped with
two perversities, $\ov{p}$ and $\ov{q}$.
 In \cite[Proposition 4.2]{CST4}, we prove the existence of a map 
 \begin{equation}\label{equa:cupprduitespacefiltre}
 -\smile -\colon\tN^i_{\ov{p}}(X;R)\otimes \tN^{j}_{\ov{q}}(X;R)\to \tN^{i+j}_{\ov{p}+\ov{q}}(X;R),
 \end{equation} 
 inducing  an associative  and  commutative graded product, called \emph{intersection cup product,}
  \begin{equation}\label{equa:cupprduitTWcohomologie}
-\smile -\colon \crH^i_{\ov{p}}(X;R)\otimes \crH^{j}_{\ov{q}}(X;R)\to
 \crH^{i+j}_{\ov{p}+\ov{q}}(X;R).
 \end{equation}
 We mention also from \cite[Propositions 6.6 and 6.7]{CST4} the existence of cap products,
\begin{equation}\label{equa:caphomology}
 -\frown - \colon \widetilde{N}_{\overline{q}}^{i}(X;R)\otimes {\gC}_{j}^{\overline{p}}(X;R)
\to {\gC}_{j-i}^{\overline{p}+\overline{q}}(X;R),
\end{equation}
such that $
(\eta\smile \omega)\frown \xi=\eta\frown(\omega\frown\xi)$. 
(By definition, we say that  the collection $\{{\gC}_{*}^{\overline{p}}(X;R)\}_{\overline{p}\in\mathcal{P}}$
is a left perverse module over the perverse algebra $\left\{ \widetilde{N}_{\overline{q}}^{*}(X;R)\right\} _{\overline{q}\in\mathcal{P}}$.)
Moreover, 
we have
\begin{equation}\label{equa:capdiff}
\gd(\omega\frown \xi)=d\omega\frown\xi+(-1)^{|\omega|}\omega\frown \gd \xi
\end{equation}
 and the cap product induces a map in homology,
\begin{equation}\label{equa:caphomologyhomology}
 - \frown - \colon \crH_{\overline{q}}^{i}(X;R)\otimes {\gH}_{j}^{\overline{p}}(X;R)
\to {\gH}_{j-i}^{\overline{p}+\overline{q}}(X;R).
\end{equation}
The map \eqref{equa:caphomology} can be extended to maps,
\begin{eqnarray}
-\frown - 
&\colon&
 \widetilde{N}_{\overline{p}}^{i}(X;R)\otimes {\gC}_{j}^{\infty,\overline{q}}(X;R)
\to {\gC}_{j-i}^{\infty,\overline{p}+\overline{q}}(X;R),\\
-\frown - 
&\colon&
 \widetilde{N}_{\overline{p},c}^{i}(X;R)\otimes {\gC}_{j}^{\infty,\overline{q}}(X;R)
\to {\gC}_{j-i}^{\overline{p}+\overline{q}}(X;R),
\end{eqnarray} 
which induce,
\begin{eqnarray}
- \frown - 
& \colon &
 \crH_{\overline{p}}^{i}(X;R)\otimes {\gH}_{j}^{\infty,\overline{q}}(X;R)
\to {\gH}_{j-i}^{\infty,\overline{p}+\overline{q}}(X;R),\\
 -\frown - 
 &\colon&
  \crH_{\overline{p},c}^{i}(X;R)\otimes {\gH}_{j}^{\infty,\overline{q}}(X;R)
\to {\gH}_{j-i}^{\overline{p}+\overline{q}}(X;R).
\end{eqnarray}

%%%%%%%%%%
\subsection{\tt A second cohomology coming from a linear dual}
Let $X$ be a stratified pseudomanifold  with a perversity $\ov{p}$.
We set $$\gC^*_{\ov{p}}(X;R)=\Hom_{R}(\gC_*^{\ov{p}}(X;R),R)$$
with the differential 
$\gd c(\xi)=-(-1)^{|c|}c(\gd \xi)$.
The homology of $\gC^*_{\ov{p}}(X;R)$ is denoted $\gH^*_{\ov{p}}(X;R)$ (or $\gH^*_{\ov{p}}(X)$ 
if there is no ambiguity) and called 
\emph{$\ov{p}$-intersection cohomology.}
From \remref{rem:nongm} and
the Universal Coefficients Theorem \cite[Theorem 7.1.4]{FriedmanBook},
we deduce that this cohomology coincides with the non-GM cohomology of \cite{FriedmanBook}.

The cap product (\ref{equa:caphomology}) defines a star map
\begin{equation}\label{equa:capcohomology}
\star\colon
{\gC}^{j}_{\overline{p}}(X;R)
\otimes
 \widetilde{N}_{\overline{q}}^{i}(X;R)
 \longrightarrow
 {\gC}^{j+i}_{\overline{p}-\overline{q}}(X;R)
\end{equation}
by
$$(c\star \omega)(\xi)=c(\omega\frown\xi).$$
We check easily
$c\star (\omega\smile \eta)=(c\star\omega)\star\eta$.
Hence, the collection $\{{\gC}^{*}_{\overline{p}}(X;R)\}_{\overline{p}\in\mathcal{P}}$
is a right perverse module over the perverse algebra 
$\left\{ \widetilde{N}_{\overline{q}}^{*}(X;R)\right\} _{\overline{q}\in\mathcal{P}}$. 
Moreover,  
we have
\begin{equation}\label{equa:stardiff}
\gd(c\star \omega)=\gd c\star \omega+(-1)^{|c|}c\star d\omega
\end{equation}
and  the star product induces
\begin{equation}\label{equa:caphomologycohomology}
 - \star - \colon {\gH}^{j}_{\overline{p}}(X;R)\otimes \crH_{\overline{q}}^{i}(X;R)
\to {\gH}^{j+i}_{\overline{p}-\overline{q}}(X;R).
\end{equation}
The module structures (\ref{equa:capcohomology}) and (\ref{equa:caphomologycohomology})
have also variants with compact supports.
We do not describe them in detail.

%%%%%%%%%%%%%%%%%%%%%%%
\subsection{\tt Background on  Poincar\'e duality}\label{sec:backPoincare} 
This notion has been described in the introduction, for compact oriented manifolds and stratified pseudomanifolds. 
We recall
the main results of \cite{CST2} and \cite{ST1} which represent a first step for a duality
over a  ring.

\begin{proposition}{\cite[Theorem B]{CST2}, \cite[Theorem B]{ST1}}\label{prop:dualityblownuphomology}
Let  $(X,\ov{p})$ be an oriented paracompact, perverse stratified pseudomanifold of dimension $n$. 
The cap product with the fundamental class $[X]\in \gH_n^{\infty,\ov{0}}(X;R)$ induces an isomorphism
$$
\crH^k_{\ov{p},c}(X;R)\xrightarrow{\cong} \gH^{\ov{p}}_{n-k}(X;R). 
$$
Moreover, if $X$ is second countable, this cap product also induces an isomorphism,
$$
\crH^k_{\ov{p}}(X;R)
\xrightarrow{\cong}
 \gH^{\infty,\ov{p}}_{n-k}(X;R).$$
\end{proposition}

%%%%%%%%%%%%
\subsection{\tt Dual of a  complex}
Let $0\rightarrow R\rightarrow QR
\xrightarrow{\rho}
QR/R\rightarrow 0$ be an injective resolution of the principal ideal domain $R$.

We denote by $I_R^*$ the cochain complex 
$\xymatrix@1{
I_R^0=QR\ar@{->>}[r]^-{\rho}
&
I^1_R=QR/R
}$
 and define the \emph{dual complex of a cochain complex}, $A^*$, as the chain complex
\begin{equation}\label{equa:dualcochain}
(DA^*)_{k}=(\Hom(A^*,I_{R}^*))_{k}=\Hom_{R}(A^k,QR)\oplus \Hom_{R}(A^{k+1},QR/R)
\end{equation}
with the differential
$\partial(\varphi_{0},\varphi_{1})= (-(-1)^k\varphi_{0}\circ d, -(-1)^k\varphi_{1}\circ d-\rho\circ \varphi_{0})$.
This dual complex verifies a universal coefficient formula, see \cite[Lemma 1.2]{MR1366538} for instance,
\begin{equation}\label{equa:dualuniversel}
\xymatrix@1{
0\ar[r]&
\Ext_{R}(H^{k+1}(A^*),R)\ar[r]&
H_{k}(DA^*)\ar[r]^-{\kappa}&
\Hom_{R}(H^k(A^*),R)\ar[r]&
0,
}
\end{equation}
where $\kappa$ is the canonical map defined by $(\kappa [\varphi_0,\varphi_1] )([a])=\varphi_0(a)$.
The complex $DA^*$ plays the same role as the Verdier dual in sheaf theory. A 
\emph{self-dual cochain complex of dimension $n$} is a complex, 
$A^{*}$, together with a quasi-isomorphism
\begin{equation}\label{equa:selfdual}
A^{\star}\to (DA^*)_{n-\star}.
\end{equation}

Similarly, we define the dual of a chain complex, $A_{*}$, as the cochain complex,
\begin{equation}\label{equa:dualchain}
(DA_{*})^{k}=(\Hom(A_{*},I_{R}^*))^{k}=\Hom_{R}(A_{k},QR)\oplus \Hom_{R}(A_{k-1},QR/R),
\end{equation}
with the differential
$d(\psi_{0},\psi_{1})= (-(-1)^k\psi_{0}\circ \partial, -(-1)^k\psi_{1}\circ \partial-\rho\circ \psi_{0})$.
This dual complex also verifies a universal coefficient formula, 
\begin{equation}\label{equa:dualuniverselhomology}
\xymatrix@1{
0\ar[r]&
\Ext_{R}(H_{k-1}(A_{*}),R)\ar[r]&
H^{k}(DA_{*})\ar[r]&
\Hom_{R}(H_{k}(A_{*}),R)\ar[r]&
0.
}
\end{equation}

%%%%%%%
\subsection{\tt Torsion and torsion free pairings}
We recall how the existence of a duality gives pairings between the torsion and torsion free parts,
see   \cite[Section 8.4]{FriedmanBook} for a similar treatment.

\begin{proposition}\label{prop:pairingdual}
Let $A^*$ and $B^*$ be two cochain complexes with finitely generated cohomology.
To any cochain map, $\cP=(\cP_{0},\cP_{1})\colon B^*\to DA^*$, sending $B^k$ to $(DA^*)_{n-k}$,
we can associate two pairings,
$$\cP_{F}\colon FH^k(B)\to \Hom(H^{n-k}(A),R)
\text{ and }
\cP_{T}\colon \Tors H^k(B)\to \Hom (\Tors H^{n-k+1}(A),QR/R).$$
The first one is defined by
$$\cP_{F}([b]([a])=\cP_{0}(b)(a)\in R.$$
 For the second one, let $[b]\in \Tors H^k(B)$. There exists $b'\in B^{k-1}$ and $\ell\in R$ such that $db'=\ell b$ and we set
$$\cP_{T}([b])=\rho\left(\frac{\cP_{0}(b')}{\ell}\right)+\cP_{1}(b) 
.$$
Moreover, the pairings $\cP_{F}$ and $\cP_{T}$ are non-singular if, and only if, $\cP$ is a quasi-isomorphism.
\end{proposition}

\begin{proof}
We first construct the following diagram,
\begin{equation}\label{equa:pairingdual}
\xymatrix{
0\ar[r]&
\Tors H^k(B)\ar[r]^-{j_{1}}\ar[d]^{\cP_{T}}&
H^k(B)\ar[r]\ar[d]^{\cP}&
\Free H^k(B)\ar[d]^{\cP_{F}}\ar[r]&
0\\
0\ar[r]&
\Hom(\Tors H^{n-k+1}(A),QR/R)\ar[r]^-{j_{2}}&
H_{n-k}(DA^*)\ar[r]^-{\kappa}&
\Hom(H^k(A),R)\ar[r]&
0.
}
\end{equation}
The upper line is the decomposition of a module in torsion and torsion free parts.
The lower one is a universal coefficient formula. Recall that the short exact sequence,
$$ 
\xymatrix{
0\ar[r]&
\Hom(A^{*+1},QR/R)\ar@{^(->}[r]&
DA^*\ar[r]&
\Hom(A^*,QR)\ar[r]&
0,
}$$
gives a long exact sequence with connecting map denoted $\delta$,
$$
\xymatrix{
\dots\ar[r]&
\Hom( H^{k+1}(A),QR)\ar[r]^-{\delta}&
\Hom( H^{k+1}(A),QR/R)\ar[r]&
H_{n-k}(DA^*)\ar[r]^-{\kappa}&
\dots
}$$
As $QR$ is injective and $H(A^*)$ finitely generated, there are  isomorphisms
$\Ker \delta\cong \Hom(H^{*+1}(A;R),R)$
and
$\Coker\delta\cong \Ext( H^{*+1}(A),R)\cong \Hom(\Tors H^{*+1}(A),QR/R)$.
Hence, the map $j_2$ is induced by the canonical inclusion
$\xymatrix@1{
\Hom(A^{*+1},QR/R)\;\ar@{^(->}[r]&
DA^*.
}$

As $\Hom(H^k(A),R)$ is torsion free, the composite
$\kappa\circ \cP\circ j_{1}$ is zero and there exists a lifting $\cP_{T}$ such that
$j_{2}\circ \cP_{T}=\cP\circ j_{1}$. This map induces $\cP_{F}$ making commutative the diagram.
The map $\cP_{F}$ is easily determined as in the statement. 

We now determine the map $\cP_{T}$.
With the notations of the statement, we analyze the compatibility of $\cP$ with the differentials. We first have:
\begin{itemize}
\item $(\cP\circ d)(b')=(\cP_{0}(db'),\cP_{1}(db'))=(\ell \cP_{0}(b),\ell \cP_{1}(b))$,
\item $\atop\begin{array}{ccl}
\partial\cP(b')&=&\partial(\cP_{0}(b'),\cP_{1}(b'))\\
&=&
(-(-1)^{n-k+1}\cP_{0}(b')\circ d,-(-1)^{n-k+1}\cP_{1}(b')\circ d -\rho\circ \cP_{0}(b')).
\end{array}$
\end{itemize}
The equality $\partial\circ \cP=\cP\circ d$ implies
\begin{equation}\label{equa:cPetdiff}
\left\{
\begin{array}{ccl}
\ell \cP_{0}(b)&=&-(-1)^{n-k+1}\cP_{0}(b')\circ d,\\
\ell \cP_{1}(b)&=&-(-1)^{n-k+1}\cP_{1}(b')\circ d -\rho\circ \cP_{0}(b').
\end{array}\right.
\end{equation}
We now show the commutation $j_{2}\circ \cP_{T}=\cP\circ j_{1}$ by proving that the difference is zero in homology:
\begin{eqnarray*}
\partial\left(\frac{\cP_{0}(b')}{\ell},0\right)
&=&
\left( -(-1)^{n-k+1}\frac{\cP_{0}(b')\circ d}{\ell},
-\rho\left(\frac{\cP_{0}(b')}{\ell}\right)\right)\\
&=&
(\cP_{0}(b),-\cP_{T}([b])+\cP_{1}(b))=\cP(b)-(0,\cP_{T}([b])).
\end{eqnarray*}
The last equality comes from the definition of $\cP_{T}$ and (\ref{equa:cPetdiff}).

If $\cP_{F}$ and $\cP_{T}$ are isomorphisms, $\cP$ is one also, from the five lemma. 
Conversely, suppose that $\cP$ is an isomorphism.
The $\Ker$-$\Coker$ exact sequence associated to (\ref{equa:pairingdual}) implies 
$\Ker \cP_{T}=\Coker \cP_{F}=0$ and $\Ker \cP_{F}\cong \Coker \cP_{T}$. But
$\Ker \cP_{F}$ is free and $\Coker \cP_{T}$ is torsion, thus
$\Ker \cP_{F}= \Coker \cP_{T}=0$.
\end{proof}
%%%%%%%%%%%%%%%%%%%%
\section{Verdier dual of intersection cochain complexes}\label{sec:capdual}

\begin{quote}
In \thmref{thm:duality}, we prove  \thmref{thm:poinclef} for any principal ideal domain as ring of coefficients. 
The main feature
is the use of a cap product which gives   the duality map between the two intersection cohomologies. 
We continue with a necessary and sufficient condition 
for the existence
of a duality at the level of the blown-up intersection cohomology (or intersection homology) itself.
\end{quote}

 \begin{proposition}\label{prop:dualitycochains}
 Let $(X,\ov{p})$ be an oriented perverse stratified pseudomanifold of dimension $n$
and $\gamma_X$ a representing cycle of  the fundamental class $[X]\in \gH_n^{\infty,\ov{0}}(X;R)$. 
The two following maps, 
$$
\crN_{\ov{p}}
\colon
\tN^{\star}_{\ov{p}}(X;R)\to (D\gC^*_{\ov{p},c}(X;R))_{n-\star},
\text{ and }
\crC_{\ov{p}}
\colon
\gC^{\star}_{\ov{p}}(X;R)\to (D\tN^{*}_{\ov{p},c}(X;R))_{n-\star},
$$
defined by 
\begin{itemize}
\item $\crN_{\ov{p}}(\omega)=(\varphi(\omega),0)$ with $\varphi(\omega)(c)=(-1)^{|\omega|\,|c|}(c\star \omega)(\gamma_{X})$,
\item $\crC_{\ov{p}}(c)=(\psi(c),0)$ with $\psi(c)(\omega)=(c\star \omega)(\gamma_{X})$,
\end{itemize}
are cochain maps.
 \end{proposition}
 
 \begin{proof} 
 Let $\rho\colon QR\to QR/R$ be the quotient map.
 We first observe that $\rho \varphi(\omega)(c)=0\in QR/R$
 since $\varphi(\omega)(c)\in R$. Also, as $\gamma_{X}$ is a cocycle,
 we have $\gd(c\star \omega)(\gamma_{X})=0$. With (\ref{equa:stardiff}), we deduce
 $$(\gd c\star \omega)(\gamma_{X})+ (-1)^{|c|}(c\star d\omega)(\gamma_{X})=0.$$
 Thus, we have
 $(-1)^{|\omega|(|c|\,+1)}\varphi(\omega)(\gd c)+(-1)^{|\omega|\,|c|}\varphi(d\omega)(c)
 =
 0$
 which implies
 $\varphi(d\omega)=-(-1)^{|\omega|}\varphi(\omega)\circ \gd$.
 From these observations, we get
 \begin{eqnarray*}
 \partial\crN_{\ov{p}}(\omega)
 &=&
 (\partial \varphi(\omega),0)=
 (-(-1)^{|\omega|}\varphi(\omega)\circ \gd, -\rho \varphi(\omega))=(\varphi(d\omega),0)
 \\
 &=&
 \crN_{\ov{p}}(d\omega).
 \end{eqnarray*}
 The proof is similar for $\crC_{\ov{p}}$.
 \end{proof}

\begin{theoremb}\label{thm:duality}
Let $(X,\ov{p})$ be a 
paracompact, separable and oriented perverse stratified pseudomanifold of dimension $n$.
 Then, the two  maps, 
$
\crN_{\ov{p}}
\colon
\tN^{\star}_{\ov{p}}(X;R)\to (D\gC^*_{\ov{p},c}(X;R))_{n-\star}$
and
$\crC_{\ov{p}}
\colon
\gC^{\star}_{\ov{p}}(X;R)\to (D\tN^{*}_{\ov{p},c}(X;R))_{n-\star}
$,
of \propref{prop:dualitycochains},
are quasi-isomorphisms.
 \end{theoremb}
 
 We need some lemmas before giving the proof. The first one is proven in  \cite{CST4}.
 
  \begin{lemma}{\cite[Lemma 13.3]{CST4}}\label{lem:bredon}
Let $ X $ be a locally compact topological space, metrizable and separable. We are given an open basis
of $ X $, $ \mathcal U = \{U_{\alpha} \} $, closed by finite intersections, and a statement $P(U) $ on  open subsets of $X$ satisfying the following three properties.
\begin{enumerate}[a)]
\item  The property $P(U_{\alpha}) $ is true for all $ \alpha $.
\item  If $ U $, $ V $ are open subsets of $ X $ for which  properties $P(U) $, $ P (V) $ and $ P (U \cap V) $ are true, then
$ P (U \cup V) $ is true.
\item If $ (U_{i})_{i \in I} $ is a family of open subsets of $ X $, pairwise disjoint, verifying the property $ P (U_{i}) $ for all $ i \in I $, then $ P (\bigsqcup_i U_{i}) $ is true.
\end{enumerate}
Then the property $ P (X) $ is true.
\end{lemma}

\begin{lemma}\label{lem:supbredon}
Suppose given a cochain map,
$\psi_{X}\colon A^*(X)\to B^*(X)$, for any  
paracompact, separable  perverse stratified pseudomanifold $X$,
satisfying the following three properties.
\begin{enumerate}[i)]
\item The map $\psi_{X}$ is a quasi-isomorphism for any $X=\R^a\times \rc L$, with $L$ a compact 
stratified pseudomanifold
or the emptyset.
\item The two complexes, $A^*(X)$ and $B^*(X)$, verify the Mayer-Vietoris property and $\psi_{X}$ induces a morphism
of exact sequences (up to sign).
\item If $\psi_{U_{i}}$ is a quasi-isomorphism for a family of disjoint stratified pseudomanifolds, 
then $\psi_{\sqcup_{i}U_{i}}$
is a quasi-isomorphism.
\end{enumerate}
Then $\psi_{X}$ is a quasi-isomorphism for any $X$.
\end{lemma}

\begin{proof}
As $X$ is  metrizable (cf. \cite[Proposition 1.11]{CST3})  we may use \lemref{lem:bredon}.
We denote by $P(X)$ the property ``$\psi_{X}$ is a quasi-isomorphism''. We consider the family
$\cU = \{ U_\alpha\}$ formed of the open subsets of charts of the topological stratified pseudomanifold $X$
together with the open subsets of the conical charts of the topological manifold $X\menos X_{n-1}$.

 Observe that Property b)$_{\cU} $ is a direct consequence of the existence
 of  a morphism between the Mayer-Vietoris sequences in the domain and codomain. 
 Also, Property c)$_{\cU}$ coincides with the hypothesis iii). We are reduced to establish a)$_{\cU}$.
 
For  that, we proceed by induction on the depth  of $X$. If $\depth\, X=0$, 
the stratified pseudomanifold  $X$ is a manifold and
we have that  $U_\alpha $ is an open subset of $\R^n$. 
 We now consider the basis $\cV$ formed of the open $n$-cubes of $\R^n$ included in $U_{\alpha}$. This family is
 closed by finite intersections and verifies the hypotheses of \lemref{lem:bredon}, Property a)$_{\cV}$ being given by 
 the hypothesis i). This proves $P(U_\alpha )$.
 
 To carry out the inductive step, we first observe that $P(U)$ is already established for each open subset $U$ of 
 $X\menos X_{n-1}$ since $\depth \,U=0$. We consider an open subset $U_{\alpha}$ of  a conical chart 
 $Y=\R^a\times \rc L$ of $X$, with $L$  a compact stratified pseudomanifold.
 We choose the basis $\cW$ of open subsets of $U_{\alpha}$, formed of the open subsets $W\subset U_{\alpha}$ with 
 $W\cap (\R^a\times \{\tv\})=\emptyset$, which are stratified pseudomanifolds with $\depth\, W \leq \depth\,
 (\R^a \times (\rc L\menos \{ \tv\}) < \depth\, Y \leq \depth\, X$,
 together with the open subsets $W=B\times \rc_{r} L\subset U_{\alpha}$, 
 where $B$ is an open $a$-cube, $r>0$ and $\rc_{r}L=(L\times [0,r[)/(L\times\{0\})$.
 The family $\cW$ is closed by finite intersections and verifies the hypotheses of \lemref{lem:bredon}, 
 the property a)$_{\cW}$ being given by induction and the hypothesis i). This proves $P(U_{\alpha})$.
\end{proof}

 The third lemma is the proof of \thmref{thm:duality}  in a particular generic case.
 
 \begin{lemma}\label{lem:duality}
 The conclusion of \thmref{thm:duality} is true if
 $X=\R^a\times \rc L$,
 where $L$ is a compact oriented perverse stratified pseudomanifold of dimension $m-1$.
 \end{lemma}
 
 \begin{proof}
 We begin by checking the finite generation of the various homologies and cohomologies. First,
 we know that the intersection homology of a compact stratified pseudomanifold is finitely generated, see
 \cite[Corollary 6.3.40]{FriedmanBook} for instance. From Poincar\'e duality, universal coefficients formula
 or direct computations, 
 this infers the finite generation of $\crH_{\ov{p}}^*(X)$, $\gH^*_{\ov{p}}(X)$ and $\gH^{\infty,\ov{p}}_{*}(X)$.
 For the blown-up cohomology with compact supports, $\crH_{\ov{p},c}^*(X)$, this is a consequence of
 \cite[Propositions 2.18 and 2.19]{CST2}. As we do not find an explicit reference for the last one, 
 $\gH^*_{\ov{p},c}(X)$, we supply a
 short direct proof. 
 
Set
 $K_{n}=[-n,n]^a\times \rc_{n}L$
 with $\rc_{n}L=L\times [0,(n-1)/n[/L\times\{0\}$. The family $(K_{n})_{n}$ being cofinal among the compact subsets of 
 $\R^a\times \rc L$, we have
 $$\gH^*_{\ov{p},c}(\R^a\times \rc L)=
 \varinjlim_{n}\,\gH^*_{\ov{p}}(\R^a\times \rc L, (\R^a\times \rc L)\menos K_{n}).$$
As all the open subsets $(\R^a\times \rc L)\menos K_{n}$ are stratified homeomorphic, it suffices to consider $n=0$
and
\begin{equation}\label{equa:conejoin}
\gH^*_{\ov{p},c}(\R^a\times \rc L)
=\gH^*_{\ov{p}}(\R^a\times \rc L, (\R^a\times \rc L)\menos \{(0,\tv)\}).
\end{equation}
As we observed before, the cohomology $\gH^*_{\ov{p}}(\R^a\times \rc L)$ is finitely generated. For the second one, 
we know that $\R^a\times \rc L\menos \{(0,\tv)\}$ is stratified homeomorphic to $\rc (S^{a-1}\ast L)\menos \{\tu\}$,
cf. \cite[5.7.4]{MR2273730} and proof of \cite[Proposition 3.7]{CST5}. 
As $\gH^*_{\ov{p}}(\rc (S^{a-1}\ast L)\menos \{\tu\})$ is also finitely generated, so is the relative homology of
(\ref{equa:conejoin}).

For the rest of this proof, we set $X=\R^a\times \rc L$.
Let us observe that the two following maps are quasi-isomorphisms,
$$\tN^{\star}_{\ov{p},c}(X)\to \gC_{n-\star}^{\ov{p}}(X)\to (D\gC^*_{\ov{p}}(X))_{n-\star}.$$
The left-hand map is the Poincar\'e duality of \cite[Theorem B]{CST2}.
For the right hand one,  this comes from the fact that $\gC_{*}^{\ov{p}}(X)$ is a free module
with finitely generated homology, 
see \cite[Proof of Proposition 1.3]{MR1366538} for instance. By applying the dual functor to this composition,
we get the map
$$\crC_{\ov{p}}
\colon
\gC^{\star}_{\ov{p}}(X)\to 
(DD\gC^*_{\ov{p}}(X))^{\star}\to
(D\tN^{*}_{\ov{p},c}(X))_{n-\star},$$
which is a quasi-isomorphism, since the homologies are finitely generated, see \cite[Proposition 1.3]{MR1366538}.

For the second quasi-isomorphism, $\crN_{\ov{p}}$, we decompose 
it as
$$\tN^{\star}_{\ov{p}}(X)\to \gC_{\star}^{\infty,\ov{p}}(X)
\to (D\gC^*_{\ov{p},c}(X))_{n-\star},$$
where the left-hand map is  the duality of \cite{ST1} (recalled in \propref{prop:dualityblownuphomology}).
Thus the proof is reduced to the study of the right-hand map.
First, recall from \cite{FM} and \cite[Proposition 2.2]{ST1}, that
$\gC^*_{\ov{p},c}(X)=\varinjlim_{K}\gC^*_{\ov{p}}(X,X\menos K)=
\gC^*_{\ov{p}}(X,X\menos \{(0,\tv)\})$
and
$\gC_{*}^{\infty,\ov{p}}(X)=\varprojlim_{K}\gC_{*}^{\ov{p}}(X,X \menos K)
=\gC^*_{\ov{p}}(X,X\menos \{(0,\tv)\})$.
Thus, it is sufficient to prove the existence of a quasi-isomorphism,
$$\gC_{\star}^{\ov{p}}(X,X \menos \{(0,\tv)\})
\to
(D\gC^*_{\ov{p}}(X,X\menos \{(0,\tv)\}))_{n-\star}.$$ 
As the complex $\gC_{k}^{\ov{p}}(X)$ is free with finitely generated homology, the evaluation map $\gC_{\star}^{\ov{p}}(X)
\to
(D\gC^*_{\ov{p}}(X))_{n-\star}$
 is a quasi-isomorphism. 
 
 Replacing the subspace $X\menos \{(0,\tv)\}$ by $\rc (S^{a-1}\ast L)\menos \{\tu\}$
as we do above, we also get a quasi-isomorphism
$\gC_{k}^{\ov{p}}(X\menos \{(0,\tv)\})
\to
(D\gC^*_{\ov{p}}(X\menos \{(0,\tv)\}))_{n-k}$.
From a five lemma argument, we get that the  map
$\gC_{\star}^{\infty,\ov{p}}(X)
\to (D\gC^*_{\ov{p},c}(X))_{n-\star}$
is a quasi-isomorphism. 
 \end{proof}
 
 \begin{proof}[Proof of \thmref{thm:duality}]
 We check the hypotheses of \lemref{lem:supbredon}. 
 \begin{enumerate}[i)]
 \item This is \lemref{lem:duality}.
\item We already know that each complex has a Mayer-Vietoris sequence. The fact that any of the
maps under consideration induces a morphism of exact sequence comes from the natutality of the
choice of the fundamental classes: for an open subset $U\subset X$,
we may choose the restriction of a fixed cycle $\gamma_{X}\in \gC^{\infty,\ov{0}}_{n}(X)$ 
representing the fundamental class 
of $X$ to define the fundamental class of $U$.
\item This is a consequence of the fact that the duality $D$ sends inductive limits to projective limits and of 
the following properties:\\
 $\tN^*_{\ov{p}}(\sqcup_{i} U_{i}) 
 = \prod_{i}\tN^*_{\ov{p}}(U_{i})$,
  $\tN^*_{\ov{p},c}(\sqcup_{i} U_{i}) 
 = \oplus_{i}\tN^*_{\ov{p},c}(U_{i})$,
  $\gC^*_{\ov{p}}(\sqcup_{i} U_{i})
  =\prod_{i}\gC^*_{\ov{p}}(U_{i})$,
  and \\
 $\gC^*_{\ov{p},c}(\sqcup_{i} U_{i})
 =\oplus_{i}\gC^*_{\ov{p},c}(U_{i})
 $.
 \item This is immediate.
\end{enumerate}
 \end{proof}

\begin{remark}
The two complexes,
$\tN^*_{\ov{\bullet}}(-)$ and $\gC^*_{\ov{\bullet}}(-)$,
have elements  of the same nature (they associate a number to chains)
but have a different behaviour.
\begin{itemize}
\item In  $\tN^*_{\ov{p}}(-)$ a blown-up cochain is defined on each filtered simplex.
\item The  cochains in $\gC^*_{\ov{p}}(-)$ are defined only on the chains of
$\ov{p}$-intersection. 
We can view them as relative cochains taking the value 0 on chains which are not of $\ov{p}$-intersection.
\end{itemize}

Viewing $\crH^*_{\ov{p}}(-)$ as an absolute cohomology and
$\gH^*_{\ov{p}}(-)$ as a relative one, the pairings coming from \thmref{thm:duality}
look like the Poincar\'e-Lefschetz non-singular pairings of
a compact oriented manifold with boundary, that is, by example for the  torsion free part,
\begin{equation*} 
\Free H^{*}(M,\partial M;R)\otimes \Free H^{n-*}(M;R)\to R.
\end{equation*}
\end{remark}

\begin{remark}
For sake of simplicity, we suppose that $X$ is an oriented compact 
stratified pseudomanifold. 
In \cite[Theorem B]{CST2}, we prove that the chain map defined by the cap product with a cycle 
$\gamma_{X}$ representing the
fundamental class, is a quasi-isomorphism,
$$-\cap \gamma_{X}\colon \tN^{\star}_{\ov{p}}(X,R)\to \gC_{n-\star}^{\ov{p}}(X,R).$$
\thmref{thm:duality} shows that the composition with a certain dualization of the chain complex,
in fact  the Verdier dual of the linear dual, is a quasi-isomorphism as well,
$$
\tN^{\star}_{\ov{p}}(X;R)\to (D\gC^*_{\ov{p}}(X;R))_{n-\star},$$
making of the blown-up cochain complex a Verdier dual of $\gC^*_{\ov{p}}(X;R)$.
In the next section, we are now looking for a  duality
involving only the blown-up cochains.
\end{remark}
%%%%%%%%%%%%%%%%%%%%%%%%%
\section{Poincar\'e duality with pairings}\label{sec:peripheralGS}

\begin{quote}
After defining the peripheral complex, we prove  two main properties of it 
(see \cite{GS} in the case of compact PL-pseudomanifolds): 
its link with the occurrence of a duality in intersection homology and the existence
of a duality on itself. 
In \propref{prop:localfreeacyclic}, we show that the locally torsion free condition, 
required by Goresky and Siegel  (see \defref{def:locallyfree})
is equivalent to a local acyclicity of the peripheral complex.
Finally, we give an example of a stratified pseudomanifold which is not locally torsion free and has 
an acyclic peripheral complex, thus satisfies Poincar\'e duality.
\end{quote}

A Poincar\'e duality on an  oriented stratified pseudomanifold, $X$, similar to the duality on  manifolds, 
should be the existence of a quasi-isomorphism
\begin{equation}\label{equa:pairingR}
\crD_{\ov{p}}\colon \tN^{\star}_{\ov{p}}(X;R)\to
D(\tN^{*}_{D\ov{p},c}(X;R))_{n-\star}.
\end{equation}
With \thmref{thm:duality}, such map  $\crD_{\ov{p}}$ 
can be obtained from the  composition of $\crC_{D\ov{p}}$
with a quasi-isomorphism
$\tN^*_{\ov{p}}(X;R)\xrightarrow{\simeq} \gC^*_{D\ov{p}}(X;R)$. We introduce now such crucial map,
already present in 
\cite{CST1} and \cite[Section~13]{CST4}.

\begin{proposition}\label{prop:chipq}
Let $(X,\ov{p})$ be a  perverse stratified pseudomanifold 
and $\varepsilon\colon (\gC_{*}^{\ov{t}}(X;R),\gd)\to (R,0)$  an augmentation.
Then there is a cochain map,
\begin{equation}\label{equa:lechi}
\chi_{\ov{p}}\colon \tN^{\ast}_{\ov{p}}(X;R)\to \gC^{\ast}_{D\ov{p}}(X;R),
\end{equation}
defined by 
$\chi_{\ov{p}}(\omega)=\varepsilon \star \omega$.
We denote by
$\chi_{\ov{p},c}\colon \tN^{\ast}_{\ov{p},c}(X;R)\to \gC^{\ast}_{D\ov{p},c}(X;R)$
the restriction of $\chi_{\ov{p}}$ to the cochains with compact supports.
\end{proposition}

\begin{proof}
Let  $\omega\in \tN^k_{\ov{p}}(X;R)$
and
$\xi\in \gC_{k}^{D\ov{p}}(X;R)$.
With the notation of the statement, we  observe from (\ref{equa:caphomology}) that
$\omega\frown\xi\in \gC_{0}^{\ov{t}}(X;R)$ and 
thus $\varepsilon(\omega\frown\xi)$ is well defined,
To check the compatibility with the differentials, we apply $\varepsilon$ at the two sides of (\ref{equa:capdiff}).
First, we have $\varepsilon(\gd(\omega\frown\xi))=0$ which implies,
\begin{eqnarray*}
0&=&
\varepsilon((d\omega)\frown \xi)+(-1)^{|\omega|}\varepsilon(\omega\frown (\gd\xi))
=
\chi_{\ov{p}}(d\omega)(\xi)+(-1)^{|\omega|}\chi_{\ov{p}}(\omega)(\gd \xi)\\
&=&
\chi_{\ov{p}}(d\omega)(\xi)-
\gd\chi_{\ov{p}}(\omega)( \xi))
\end{eqnarray*}
and $\gd\chi_{\ov{p}}(\omega)=\chi_{\ov{p}}(d\omega)$.
\end{proof}

\begin{corollary}\label{cor:todoben}
Let $(X,\ov{p})$ be a  
paracompact, separable and oriented perverse stratified pseudomanifold of dimension $n$
and $\gamma_X$ a representing cycle of  the fundamental class $[X]\in H_n^{\infty,\ov{0}}(X;R)$. 
Then, the map
\begin{equation}\label{equa:pairingR2}
\crD_{\ov{p}}\colon \tN^{\star}_{\ov{p}}(X;R)\to
(D\tN^*_{D\ov{p},c}(X;R))_{n-\star} 
\end{equation}
defined by
$\crD_{\ov{p}}(\omega)(\omega',\omega'')=((\varepsilon\star\omega\star\omega')(\gamma_{X}),0)$ is a quasi-isomorphism
 if, and only if, the map $D\chi_{D\ov{p},c}$ is one also.
\end{corollary}

The torsion and torsion free pairings arising from $\crD_{\ov{p}}$ are studied in \secref{sec:periphmas}.

\begin{proof}
With the notation of \propref{prop:dualitycochains},
the map $\crD_{\ov{p}}$ is equal to the following composition,
$$\tN^{\star}_{\ov{p}}(X)\xrightarrow{\crN_{\ov{p}}}
(D  \gC^*_{\ov{p},c}(X))_{n-\star}\xrightarrow{D\chi_{D\ov{p},c}}
(D\tN^*_{D\ov{p},c}(X))_{n-\star}.$$
Thus the result is a consequence of   \thmref{thm:duality}.  
Let us also notice that $\crD_{\ov{p}}= \crC_{D\ov{p}} \circ \chi_{\ov{p}}$.
\end{proof}

In view of \corref{cor:todoben}, the cofibers  of $\chi_{\ov{p}}$ and $\chi_{\ov{p},c}$
in the category of cochain complexes play a fundamental role
in Poincar\'e duality. We call them
the \emph{peripheral complexes.} (A brief analysis  shows that they
correspond  to the global sections of the peripheral sheaf of \cite{GS}.)

\begin{definition}\label{def:peripheral}
Let $(X,\ov{p})$ be a perverse stratified pseudomanifold. 
The \emph{$\ov{p}$-peripheral complex} of $X$ is the mapping cone of 
$\chi_{\ov{p}}\colon \tN^*_{\ov{p}}(X;R)\to \gC^*_{D\ov{p}}(X;R)$; i.e.,
$$R^*_{\ov{p}}(X;R)=(\gC^{*}_{D\ov{p}}(X;R)\oplus \tN^{*+1}_{\ov{p}}(X;R),D),
\text{ with }
D(c,\omega)=(dc+\chi_{\ov{p}}(\omega),-d\omega).$$
We denote by  $\crR^*_{\ov{p}}(X;R)$ the homology of $R^*_{\ov{p}}(X;R)$
and call it the \emph{peripheral $\ov{p}$-intersection 
cohomology of $X$}.
Similarly, we define $R^*_{\ov{p},c}(X;R)$ and $\crR^*_{\ov{p},c}(X;R)$ from $\chi_{\ov{p},c}$.
\end{definition}

If $R$ is a field, the maps $\chi_{\ov{p}}$ and $\chi_{\ov p,c}$ are  quasi-isomorphisms,
see \cite[Theorem F]{CST4} and \cite[Proposition 2.23]{CST2}.
Therefore, the  \emph{peripheral cohomologies} $\crR^*_{\ov{p}}(X;R)$ and $\crR^*_{\ov{p},c}(X;R)$
\emph{are $R$-torsion.} 
Also, from a classical argument, as $\tN^*_{\ov{p}}(-;R)$ and $\gC^*_{\ov{p}}(-;R)$ have
  Mayer-Vietoris exact sequences, so does  the peripheral complex.

The next result concerns the  existence of a duality on the peripheral cohomology, $\crR^*_{\bullet}(-;R)$, 
we follow the same way  as in \cite[Proposition 9.3]{GS}. (Let us also notice that this technique works  
in the general framework of a triangulated category, see \cite[Theorem 1.6]{MR1763933}.)
 
 \begin{theoremb}[\cite{GS}]\label{thm:dualityperiph}
 Let $(X,\ov{p})$ be a 
 paracompact, separable and oriented perverse stratified pseudomanifold of dimension $n$
and $\gamma_X$ a representing cycle of  the fundamental class $[X]\in H_n^{\infty,\ov{0}}(X;R)$. 
 Then, there is a cochain map,
 $$\varphi_{\ov{p}}\colon R^{\star}_{\ov{p}}(X;R)\to (D R^*_{D\ov{p},c}(X;R))_{n-1-\star},$$
 inducing an isomorphism in homology. 
 \end{theoremb}

\begin{proof}
 The various arrows of the following diagram are specified below.
\begin{equation}\label{equa:bigdiag}
\xymatrix@-4ex{
&R^k_{\ov{p}}(X)\ar[ld]_-{[1]}\ar'[d][dd]^(.4){\varphi_{\ov{p}}}\\ 
\tN^k_{\ov{p}}(X)\ar[rr]^(.44){\chi_{\ov{p}}}\ar[dd]_-{\crN_{\ov{p}}}&&
\gC^k_{D\ov{p}}(X)\ar[lu]\ar[dd]^-{\crC_{D\ov{p}}}\\
&(D R^*_{D\ov{p},c}(X))_{n-k-1}\ar[ld]&\\
(D\gC^{n-k}_{\ov{p},c}(X))_{n-k}\ar[rr]_-{\chi^{\sharp}_{D\ov{p},c}}&&
(D\tN^*_{D\ov{p},c}(X))_{n-k}\ar[ul]^{-[1]}
}
\end{equation}
The map $\chi_{\ov{p}}$ is  recalled in (\ref{equa:lechi}) and 
$\chi^{\sharp}_{D\ov{p},c}=D \chi_{D\ov{p},c}$ is defined by  duality. 
The two vertical maps of the front square are defined in \propref{prop:dualitycochains}. 
By construction, the front square commutes and induces the cochain map $\varphi_{\ov{p}}$.  
From \thmref{thm:duality} and the 5-lemma, we get that
 $\varphi_{\ov{p}}$ induces an isomorphism. 
\end{proof}

\begin{corollary}\label{cor:pdp}
 Let $(X,\ov{p})$ be a  
 paracompact, separable and oriented perverse stratified pseudomanifold.
 Then, the following conditions are equivalent.
 \begin{enumerate}
 \item The stratified pseudomanifold $(X,\ov{p})$ verifies Poincar\'e duality; i.e.,
 the map $\crD_{\ov{p}}$ is a quasi-isomorphism.
  \item The map  $D\chi_{D\ov{p},c}$ is a quasi-isomorphism.
 \item The map $\chi_{\ov{p}}$ is a quasi-isomorphism
 \end{enumerate}
\end{corollary}

\begin{proof}
The equivalence of (1) and (2) is done in \corref{cor:todoben}
and the equivalence of (2) and (3) comes from the commutativity of the front face of 
(\ref{equa:bigdiag}) and \thmref{thm:duality}.
\end{proof}

This corollary means that $\crD_{\ov{p}}$ is a quasi-isomorphism if, and only if, 
the peripheral complex $R^k_{\ov{p}}(X;R)$ is acyclic. In \cite{GS}, Goresky and Siegel give a sufficient condition 
of acyclicity for the peripheral complex that we describe now.

First, let us observe that the two complexes, 
$ \tN^*_{\ov{p}}(X;R)$ and $\gC^*_{D\ov{p}}(X;R)$,
are connected by a cochain map,
have  Mayer-Vietoris sequences, coincide on Euclidean spaces and have the same behaviour 
for disjoint union of open subsets. 
Therefore (see \lemref{lem:supbredon}), the map $\chi_{\ov{p}}$ induces an isomorphism 
if it does on the products $\R^n\times \rc L$ where $L$ is a compact stratified pseudomanifold.
To exemplify this point, we first reduce to the particular case of a cone over a compact manifold, $X=\rc M$.
Already known computations (see \examref{exam:cone}) 
show that in this case, the difference between the two cohomology groups
is concentrated in one degree, where we have
\begin{equation}\label{equa:2cones}
\crH^{\ov{p}(\tv)+1}_{\ov{p}}(\rc M;R)= 0
\text{ and }
\gH^{\ov{p}(\tv)+1}_{D\ov{p}}(\rc M;R)=\Tors \gH_{\ov{p}(\tv)}(M;R).
\end{equation}
Thus the lack of torsion in the homology of the manifold $M$, in this critical degree, is a necessary and sufficient condition
for having an isomorphism 
between the two cohomologies $\crH^*_{\ov{p}}(\rc M;R)$ and $\gH^*_{D\ov{p}}(\rc M;R)$. 
We examine now the general case.

First, observe that 
``the'' link of a stratum is not uniquely determined but all  the links of points lying in the same stratum have 
isomorphic intersection homology groups,
see \cite[Corollary 5.3.14]{FriedmanBook}.
Thus, for sake of simplicity, we use the expression \emph{the link $L_{S}$ of a stratum $S$} if 
only the intersection homology groups of the links appear, as in the following definition.

\begin{definition}[\cite{GS}]\label{def:locallyfree}
A stratified pseudomanifold $X$ is \emph{locally $(D\ov{p},R)$-torsion free} 
if
\begin{equation}\label{equa:locallyfree}
\Tors\, \gH^{\ov{p}}_{D\ov{p}(S)}(L_{S};R)=0,
\end{equation}
for each stratum $S$ with associated link $L_{S}$. 
\end{definition}

As $\dim X=\dim L_{S}+\dim S+1$, we have $D\ov{p}(S)=\ov{t}(S)-\ov{p}(S)=\dim L_{S}-\ov{p}(S)-1$.
From Poincar\'e duality, one can deduce (see for instance \cite[Corollary 8.2.5]{FriedmanBook}) that
$X$ is locally $(\ov{p},R)$-torsion free if, and only if, it is locally $(D\ov{p},R)$-torsion free. 
We therefore use them indifferently. Let us also notice that any open subset of a
locally $(\ov{p},R)$-torsion free stratified pseudomanifold is a
locally $(\ov{p},R)$-torsion free stratified pseudomanifold.

\begin{proposition}\label{prop:chi-iso}
Let $(X,\ov p)$ be a paracompact, separable, perverse, stratified pseudomanifold. 
If  $X$ is   locally 
$(\ov{p},R)$-torsion free,
then the maps $\chi_{\ov{p}}$ and $\chi_{\ov{p},c}$ induce isomorphisms,
\begin{equation}\label{equa:gmtw}
\chi^*_{\ov{p}}\colon \crH^*_{\ov{p}}(X;R)\xrightarrow{\cong} \gH^*_{D\ov{p}}(X;R)
\text{ and }
\chi^*_{\ov{p},c}\colon \crH^*_{\ov{p},c}(X;R)\xrightarrow{\cong} \gH^*_{D\ov{p},c}(X;R).
\end{equation}
\end{proposition}

\begin{proof}
The assertion for $\chi^*_{\ov{p}}$ is proven in  \cite[Theorem F]{CST4}
and in \cite[Proposition 2.23]{CST2} for $\chi^*_{\ov{p},c}$.
\end{proof}
By using that the dual of a quasi-isomorphism is a quasi-isomorphism and \corref{cor:pdp}, 
we deduce that
a locally $(\ov{p},R)$-torsion free stratified pseudomanifold satisfies Poincar\'e duality and
we recover \cite[Theorem 4.4]{GS}.
The reverse way is not true in general, as  \examref{exam:periphnofree} shows.

\begin{proposition}\label{prop:examdualnotfree}
There are examples of compact oriented stratified pseudomanifolds with a perversity $\ov{p}$, which are not
locally $(\ov{p},R)$-torsion free and whose $\ov{p}$-intersection homology  satisfies  Poincar\'e duality.
\end{proposition}

We complete this section with a characterization of the property 
``$(\ov{p},R)$-torsion free'' in terms of local
acyclicity of the peripheral complex, which
is equivalent to the nullity of the associated sheaf, considered in \cite{GS}.

\begin{proposition}\label{prop:localfreeacyclic}
Let $X$ be a compact oriented stratified pseudomanifold of dimension $n$
and $\ov{p}$ a perversity. Then, the stratified pseudomanifold $X$ is locally $(\ov{p},R)$-torsion free if, and only if,
$\crR^*_{\ov{p}}(U;R)=0$ for any open subset $U\subset X$.
\end{proposition}

\begin{proof} %[Proof of \propref{prop:localfreeacyclic}]
Suppose $\crR^*_{\ov{p}}(U)=0$ for any open subset $U$ of $X$. We choose a conical chart
$U=\R^{n-k}\times \rc L$. 
From \examref{exam:cone}, we observe   that the condition $\crR^*_{\ov{p}}(U)=0$
implies
$$\crR^*(\rc L)=\Tors \gH^{D\ov{p}}_{\ov{p}(\tv)}(L)=0.
$$
Therefore, the stratified pseudomanifold $X$ is locally $(\ov{p},R)$-torsion free.

We establish now the reverse way and suppose that
the stratified pseudomanifold $X$ is locally $(\ov{p},R)$-torsion free.
We apply \lemref{lem:bredon} taking for $P(U)$ the property

\centerline{``for any open subset $V$ of $U$, we have $\crR^*_{\ov{p}}(V)=0$.''}

We proceed by induction on the depth of the stratified pseudomanifold, starting easily with the case of
a manifold with empty singular set. The induction uses two steps.

$\bullet$ First, we prove $P(U)$ for any open subset $U$ of a fixed conical chart $Y=\R^m\times \rc L$. This is obvious if $L=\emptyset$
therefore, we suppose $L\neq \emptyset$.
We consider the following basis, $\cV$, of open subsets $V$ of $U$ composed of subsets of two kinds:
\begin{itemize}
\item[--] The open subsets $V$ of $U$ that do not contain the apex of $\rc L$. 
They are stratified pseudomanifolds of depth
less than $\depth\, X$ and the induction hypothesis can be used.
\item[--]  The open subsets $V=B\times \rc_{\varepsilon} L$, where $B\subset \R^m$ is an open cube,
$\varepsilon>0$ and $\rc_{\varepsilon}L=(L\times [0,\varepsilon[)/(L\times \{0\})$. The
acyclicity of $R^*_{\ov{p}}(V)$ comes from
the local $(D\ov{p},R)$-torsion freeness of $X$, as at the beginning of this proof.
\end{itemize}
This family $\cV$ is closed for finite intersections and satisfies the hypotheses of \lemref{lem:bredon}. 
We have just proved  
condition a). Property b) is a consequence of the existence of Mayer-Vietoris sequences and c) is direct.
Thus, $P(U)$ is true.

$\bullet$ Finally, for establishing the property $P(X)$,
 we choose the open basis composed of open subsets of conical charts 
 or regular open subsets and
apply \lemref{lem:bredon}. Note that condition a) is proved in the first step. 
For b) and c), the  arguments used for a conical chart apply also.
\end{proof}

Note that \examref{exam:periphnofree} is in accordance with \propref{prop:localfreeacyclic}. 
Here, conical charts are products,
$]0,1[\times  \rc (S^1\times S^1\times \R P^3)$,
that are not locally torsion free.
%

%%%%%%%%%%%%%%%%%%
\section{A relative complex}\label{sec:relativeDp}
\begin{quote}
In this section, we take over the relative complex introduced by Friedman and Hunsicker
\cite{MR3028755}, for locally torsion free compact PL-pseudomanifolds. 
We  extend the properties given in loc. cit.  to the case
of an acyclic peripheral complex, with coefficients in $R$.
\end{quote}

Let $(X,\ov{p})$ be a perverse space such that $\ov{p}\leq D\ov{p}$.
We consider the homotopy cofiber sequence,  
$$\tN_{\ov{p}}^*(X;R)\to \tN_{D\ov{p}}^*(X;R)\to \pN(X;R).$$
We call it the $(D\ov{p},\ov{p})$-relative complex (or relative complex if there is no ambiguity) and
denote its homology by $\pH^*(X;R)$. Similarly, we consider the homotopy cofiber sequence
$$\gC^*_{D\ov{p},c}(X;R)\to \gC^*_{\ov{p},c}(X;R)\to \pC(X;R).$$

In \cite[Lemma 3.7]{MR3028755}, G. Friedman and E. Hunsicker also introduce  relative complexes
for intersection homology with rational coefficients of compact PL-pseudomanifolds.
  Their  general purpose is the extension
of Novikov additivity and Wall non-additivity in the case $R=\Q$, for
$4n$-dimensional PL-pseudomanifolds.

\begin{proposition}\label{prop:petitpasDpp}
Let $(X,\ov{p})$ be a 
paracompact, separable and oriented perverse stratified pseudomanifold of dimension $n$
with $\ov{p}\leq D\ov{p}$. 
Then there is a quasi-isomorphism
$$\psi_{\ov{p}}\colon \pN^{\star}(X;R)\to (D\pC(X;R))_{n-\star-1}.$$
\end{proposition}

\begin{proof}
We introduce  a diagram, as in the proof of \thmref{thm:dualityperiph},
\begin{equation}\label{equa:bigdiag2}
\xymatrix@-3ex{
&\pN^k(X)
\ar[ld]_-{[1]} 
\ar'[d][dd]^(-.3){\psi_{\ov{p}}}
\\ 
\tN^k_{\ov{p}}(X)\ar[rr]
\ar[dd]_-{\crN_{\ov{p}}}
&&
\tN^k_{D\ov{p}}(X)\ar[lu]\ar[dd]^{\crN_{D\ov{p}}}
\\
&(D \pC(X))_{n-k-1}\ar[ld]&\\
(D\gC^{*}_{\ov{p},c}(X))_{n-k}
\ar[rr]
&&
(D\gC^*_{D\ov{p},c}(X))_{n-k}.\ar[ul]_-{-[1]}
}
\end{equation}
The two vertical maps of the front face are quasi-isomorphisms. 
They induce, the back vertical arrow,
$\psi_{\ov{p}}$, which is  also a quasi-isomorphism.
\end{proof}

By construction, we have
$\psi_{\ov{p}}(\omega)(c,c')=((dc\star \omega+(-1)^{|c|}c\star d\omega)(\gamma_X),0)$,
where $\gamma_{X}$ is a cycle representing the fundamental class.

\begin{corollary}\label{cor:petitpasDpp}
Let $(X,\ov{p})$ be a 
paracompact, separable and oriented perverse stratified pseudomanifold of dimension $n$
such that $\chi_{\ov{p}}$ and $\chi_{\ov{p},c}$ are quasi-isomorphisms.  
Denote by $\pNc^*(X;R)$ the cofiber of 
$\tN_{\ov{p},c}^*(X;R)\to \tN_{D\ov{p},c}^*(X;R)$.
Then there is a quasi-isomorphism
$$\Psi_{\ov{p}}\colon \pN^{\star}(X;R)\to (D\pNc^*(X;R))_{n-\star-1}.$$
\end{corollary}

As in \propref{prop:pairingdual}, such quasi-isomorphism induces non-singular pairings for the torsion and the torsion free
parts of the
homology of $\pN^*(X;R)$. In the compact PL-case, the previous statement corresponds to the duality obtained in
\cite{MR3028755}.
In contrast with the peripheral complex of the previous section, 
the homology of this relative complex is not entirely torsion, see \examref{Dpp}.

\begin{proof}
Let us observe that the maps $\chi_{\ov{p},c}$ and $\chi_{D\ov{p},c}$ induce a map
$$\chi_{D\ov{p}/\ov{p},c} \colon \pNc^*(X;R)\to \pC(X;R).$$
As $\chi_{\ov{p},c}$ is a quasi-isomorphism, its dual $D\chi_{\ov{p},c}$ is one also. 
On the other hand, with \corref{cor:pdp}, as $\chi_{\ov{p}}$ is a quasi-isomorphism, 
then $D\chi_{D\ov{p},c}$  is one also. 
Therefore, with the five lemma, we deduce that $D\chi_{D\ov{p}/\ov{p},c}$
 is a quasi-isomorphism.
In conclusion, the composition 
$D\chi_{D\ov{p}/\ov{p},c} \circ \psi_{\ov{p}}$
is the quasi-isomorphism $\Psi_{\ov{p}}$.
\end{proof}

%%%%%%%%%%%%%%%%%
\section{Components of the peripheral complex}\label{sec:periphmas}

\begin{quote}
In this section, we study the peripheral complex in the compact case.
It is constituted of ``three components" coming from the torsion  and torsion free parts
of the two cohomologies defining it.  
They correspond to failures of the existence of
non-singular  torsion  or  torsion free Poincar\'e pairings for intersection homology and blown-up cohomology.
\end{quote}

 If $X$ is compact, the map 
 $\crD_{\ov{p}}\colon \tN^{k}_{\ov{p}}(X;R)\to
(D\tN^*_{D\ov{p}}(X;R))_{n-k} $ 
 generates two pairings,
\begin{equation}\label{equa:freepairingR}
\Phi_{\ov{p}} \colon  \Free\crH^k_{\ov{p}}(X;R)\otimes \Free \crH^{n-k}_{D\ov{p}}(X;R)\to R
\end{equation}
and
\begin{equation}\label{equa:torpairingR}
L_{\ov{p}} \colon 
\Tors \crH^k_{\ov{p}}(X;R)\otimes \Tors \crH^{n+1-k}_{D\ov{p}}(X;R)\to QR/R. 
\end{equation}
For sake of simplicity, we call $\Phi_{\ov{p}}$ the \emph{Poincar\'e torsion free pairing} 
and $L_{\ov{p}}$ the
\emph{Poincar\'e torsion pairing.}
Let us also observe that any of the isomorphisms of \propref{prop:dualityblownuphomology}  allows the replacement of 
$\crH^k_{\ov{p}}(-)$ by $\gH_{n-k}^{\ov{p}}(-)$,
giving pairings of the intersection homology itself.
If $\chi_{\ov{p}}$ is a quasi-isomorphism, these two pairings are non-singular. In this section,
we  are looking for sufficient conditions suitable for one of them to be non-singular.
 
\subsection{\tt Components of the peripheral complex}
From 
$\chi^*_{\ov{p}}\colon \crH^*_{\ov{p}}(X;R)\to \gH^*_{D\ov{p}}(X;R)$, we construct, by restriction and 
projection, a morphism of exact sequences,
\begin{equation}\label{equa:chiFT}
\xymatrix{
0\ar[r]&
\Tors \crH^*_{\ov{p}}(X;R)\ar[r]\ar[d]_{\chi_{\ov{p},T}^*}&
 \crH^*_{\ov{p}}(X;R)\ar[r]\ar[d]^{\chi^*_{\ov{p}}}&
 F\crH^*_{\ov{p}}(X;R)\ar[r]\ar[d]^{\chi^*_{\ov{p},F}}&
 0\\
 0\ar[r]&
 \Tors \gH^*_{D\ov{p}}(X;R)\ar[r]&
 \gH^*_{D\ov{p}}(X;R)\ar[r]&
 F \gH^*_{D\ov{p}}(X;R)\ar[r]&
 0.
}\end{equation}
As $\chi^*_{\ov{p}}\otimes QR$ is an isomorphism, the map $\chi^*_{\ov{p},F}$ is injective
and $\coker \chi^*_{\ov{p},F}$ is entirely torsion. 
Therefore, we can define. %
\begin{eqnarray*}
\crF^*_{\ov{p}}(X;R)
&=&
\coker \chi^*_{\ov{p},F}\\
\crT^*_{\ov{p},{\mathrm C}}(X;R)
&=&
\coker \chi^*_{\ov{p},\Tors}
\;
\text{ and }
\crT^*_{\ov{p},{\mathrm K}}(X;R)
=
 \ker \chi^*_{\ov{p},\Tors} \cong  \ker \chi^*_{\ov{p}}.
\end{eqnarray*}
As first observation, we deduce from
the Ker-Coker Lemma applied to (\ref{equa:chiFT}) the short exact sequences,
\begin{equation}\label{equa:cokerexact}
\xymatrix@1{ 
0\ar[r]&
\crT^*_{\ov{p},{\mathrm C}}(X;R)\ar[r]&
\coker \chi^*_{\ov{p}}\ar[r]&
\crF^*_{\ov{p}}(X;R) \ar[r]&
0.}
\end{equation}
By definition of the peripheral complex, we also have short exact sequences,
\begin{equation}\label{equa:componentperiph}
0\to
\coker \chi^*_{\ov{p}}\to
\crR^*_{\ov{p}}(X;R)\to
\ker \chi^{*+1}_{\ov{p}}\to
0.
\end{equation}
Observe from these two series of sequences that
$\crT^*_{\ov{p},{\mathrm C}}(X;R)$ is a submodule of $\crR^*_{\ov{p}}(X;R)$.

\begin{proposition}\label{suite}
Let $(X,\ov p)$ be a compact perverse stratified pseudomanifold. Then, there exists an exact sequence:
\begin{equation}\label{equa:cokerexact2}
0\to
\crF^*_{\ov{p}}(X;R)
\to
\crR^*_{\ov p} (X;R)/\crT^*_{\ov{p},{\mathrm C}}(X;R)
\to
\crT^{*+1}_{\ov{p},{\mathrm K}}(X;R)
\to
0.
\end{equation}
\end{proposition}

\begin{proof}
The proof follows directly from the commutative diagram of exact sequences,
 \begin{equation}\label{diag}
 \footnotesize
 \xymatrix{ &0&0&0&\\
0 \ar[r] 
&
\crF^*_{\ov{p}}(X;R)\ar[r] \ar[u]
&
 \crR^*_{\ov{p}}(X;R)/ \crT^*_{\ov{p},{\mathrm C}}(X;R)\ar[r] \ar[u]
&
\crT^{*+1}_{\ov{p},{\mathrm K}}(X;R) \ar[r]\ar[u]&0
\\
0 \ar[r] 
&
\coker \chi_{\ov p}\ar[u]\ar[r] 
&
\crR^*_{\ov{p}}(X;R)\ar[u] \ar[r] 
&
\crT^{*+1}_{\ov{p},{\mathrm K}}(X;R)\ar[r]\ar@{=}[u]
&0
\\
0 \ar[r] 
&
\crT^*_{\ov{p},{\mathrm C}}(X;R)\ar@{=}[r] \ar[u]
&
\crT^*_{\ov{p},{\mathrm C}}(X;R)\ar[r] \ar[u]
&0\ar[r]\ar[u]&0
\\
 & 0 \ar[u]& 0\ar[u] &0\ar[u] & }
\end{equation}
where the first column is (\ref{equa:cokerexact})
and the middle row (\ref{equa:componentperiph}).
\end{proof}

We continue by establishing the existence of a non-singular pairing between the two components coming from the restriction of $\chi^*_{\ov{p}}$ to the torsion submodules,
$\crT^*_{\ov{p},{\mathrm C}}(X;R)$
and
$\crT^*_{\ov{p},{\mathrm K}}(X;R)$.

\begin{proposition}\label{prop:torsioncompdual}
Let $(X,\ov p)$ be an oriented  compact perverse stratified pseudomanifold. Then, there is a non-singular pairing,
$$\crK_{\ov p} \colon	
\crT^k_{\ov{p},{\mathrm K}}(X;R) \otimes \crT^{n+1-k}_{D\ov{p},{\mathrm C}}(X;R)\to QR/R.$$
\end{proposition}

\begin{proof}
Consider the following commutative diagram, whose columns are exact sequences
$$\footnotesize
\xymatrix{
0\ar[d]&&
0\ar[d]\\
\crT^k_{\ov{p},{\mathrm K}}(X;R) 
\ar[rr]\ar[d]&&
\Hom( \crT^{n+1-k}_{D\ov{p},{\mathrm C}}(X;R),QR/R)
\ar[d]\\
\Tors \crH^k_{\ov{p}}(X;R)
\ar[rr]^-{\gD'_{\ov{p},\Tors}}_-{\cong}\ar[d]_{\chi^*_{\ov{p},T}}&&
\Hom(\Tors \gH^{n+1-k}_{\ov{p}}(X;R),QR/R)
\ar[d]^{(\chi^*_{D\ov{p},T})^{\sharp}}\\
\Tors \gH^k_{D\ov{p}}(X;R)
\ar[rr]^-{\gD''_{\ov{p},\Tors}}_-{\cong}\ar[d]&&
\Hom(\Tors \crH^{n+1-k}_{D\ov{p}}(X;R),QR/R)
\ar[d]\\
  \crT^{k}_{\ov{p},{\mathrm C}}(X;R)
 \ar[d]\ar[rr]&&
\Hom(\crT^{n+1-k}_{D\ov{p},{\mathrm K}}(X;R),QR/R)
\ar[d]\\
0&&0.
}$$
 Above, the maps $\gD'_{\ov{p},\Tors}$ and $\gD''_{\ov{p},\Tors}$ are the isomorphisms of the 
 torsion pairing associated via  \propref{prop:pairingdual} to the dualities
$\crN_{\ov{p}}$ and $\crC_{D\ov{p}}$  of \thmref{thm:duality}. 
The left-hand column is exact by construction and the right-hand one also, since $QR/R$ is injective. 
As $\gD'_{\ov{p},\Tors}$
and $\gD''_{\ov{p},\Tors}$ are isomorphisms, the result follows.
\end{proof}

%%%%%%%%%%%%%
\subsection{\tt Poincar\'e torsion and torsion free pairings}

As  observed before,   if the peripheral intersection cohomology vanishes, we
have  two non-singular pairings 
(\ref{equa:freepairingR}) and (\ref{equa:torpairingR}).
We study now the existence of one of these two dualities, independently of the other one.
\propref{prop:pairingdual}, \corref{cor:todoben} and \corref{cor:pdp}
give directly the following observations.

\begin{proposition}\label{prop:freetorsionpairinghomology}
Let $X$ be a compact oriented stratified pseudomanifold of dimension $n$ and $\ov{p}$ be a perversity.
\begin{enumerate}[1)]
\item The non-degenerate torsion free pairing
 \begin{equation} \nonumber
\Phi_{\ov{p}}\colon
F\crH_k^{\ov{p}}(X;R)\otimes F\crH_{n-k}^{D\ov{p}}(X;R)\to R
\end{equation}
is non-singular if, and only if, $\crF^*_{\ov{p}}(X;R) =\Coker \chi^*_{\ov{p},F}=0$. 
\item The torsion pairing
 \begin{equation} \nonumber
L_{\ov{p}}\colon
\Tors \crH_k^{\ov{p}}(X;R)\otimes \Tors \crH_{n+1-k}^{D\ov{p}}(X;R)\to QR/R
\end{equation}
can be degenerate and is non-singular if, and only if, 
$\crT^*_{\ov{p},{\mathrm C}}(X;R)
= \crT^*_{\ov{p},{\mathrm K}}(X;R)=0$,
which is also equivalent to 
$\crR^*_{\ov{p}}=\crF^*_{\ov{p}}(X;R)=\coker \chi^*_{\ov{p},F}$.
\end{enumerate}
\end{proposition}

\begin{remark} 
In the previous statement, the different possibilities can occur.
\begin{itemize}
\item In \examref{exam:susprp3cie}, the torsion free pairing is non-singular and the torsion pairing is degenerate.
\item In \examref{exam:thoms2}, the torsion free pairing is singular and the torsion pairing is non-singular.
\item In \examref{exam:pastrivial}, the torsion free pairing is singular and the torsion pairing is degenerate.
\end{itemize}
\end{remark}

%%%%%%%%%%%%%%%%%%%
 %%%%%%%%%%%%%%%%
 \subsection{\tt Poincar\'e duality for intersection homology}
Let $X$ be a compact oriented PL-pseudomanifold of dimension $n$ and $\ov{p}$ be a GM-perversity.
As recalled in the introduction, Goresky and MacPherson (\cite{GM1}) proved that the intersection pairing
defined on the $\ov{p}$-intersection homology,
$$
\pitchfork \colon H^{\ov p}_* (X;\Q) \otimes H^{D\ov p}_{n-*} (X;\Q) \to \Q,
$$
is non-singular. This duality has been extended over $\Z$  in (\ref{equa:free}) and (\ref{equa:torsion}) 
by Goresky and Siegel
(\cite{GS})  in the case locally $\ov{p}$-torsion free.

The existence (\cite{CST2}) of the isomorphism
$-\frown [X]\colon \crH^k_{\ov{p}}(X;R)\xrightarrow{\cong} H_{n-k}^{\ov{p}}(X;R)$
allows the definition of an ``intersection product'' defined on the intersection homology of a topological stratified 
pseudomanifold
from the commutativity of the following diagram.
$$\xymatrix{
\crH^k_{\ov{p}}(X;R)\otimes \crH^{\ell}_{\ov{q}}(X;R)
\ar[r]^-{-\cup-}\ar[d]^{\cong}\ar[d]_{-\frown [X]\otimes -\frown [X]}&
\crH^{k+\ell}_{\ov{p}+\ov{q}}(X;R)
\ar[d]^{-\frown [X]}\ar[d]_{\cong}
\\
H_{n-k}^{\ov{p}}(X;R)\otimes H_{n-\ell}^{\ov{q}}(X;R)
\ar[r]^-{-\pitchfork-}&
H_{n-k-\ell}^{\ov{p}+\ov{q}}(X;R).
}$$
With this structure, the map $\crD_{\ov{p}}$ of (\ref{equa:pairingR2}) and \propref{prop:pairingdual} give two pairings:
\begin{equation}\label{equa:dualgDF}
 \crG_{\ov{p},\Free}\colon \Free H_k^{\ov{p}}(X;R)\otimes \Free H_{n-k}^{D\ov{p}}(X;R)\to R
 \end{equation}
 and
 \begin{equation}\label{equa:dualgDT}
 \crG_{\ov{p},T}\colon \Tors H_k^{\ov{p}}(X;R)\otimes \Tors H_{n-k-1}^{D\ov{p}}(X;R)\to QR/R,
 \end{equation}
which are non-singular if, and only if, the peripheral complex is acyclic. The previous results on components
of the peripheral cohomology can also be translated here through the duality map $-\frown [X]$. In particular,
the torsion pairing may be non-singular even if the stratified pseudomanifold $X$ is not $(D\ov p,R)$-locally torsion free (see \examref{exam:periphnofree}) or even if  the peripheral term $\crR_{\ov p}^*(X;R)$
  is not trivial (see \examref{exam:thoms2}).

%%%%%%%%%%%%%%%%%%%%%%

\section{Examples}\label{sec:examples}

\begin{quote}
This section contains references and details on the examples  appearing in the text.
The most significant example  is \examref{exam:periphnofree} which presents a compact stratified pseudomanifold
which is not locally $\ov{p}$-torsion free but whose 
intersection homology satisfies Poincar\'e duality.
\end{quote}

In the case of isolated singularities
on an $n$-dimensional stratified pseudomanifold, $n\geq 2$,  a GM-perversity $\ov{p}$
is defined by the natural number $\ov{p}(n)=k$; we  denote  it by $\ov{k}$.

\begin{example}[\emph{Cone on a pseudomanifold}]\label{exam:cone}
Let $L$ be an $(n-1)$-dimensional compact stratified  pseudomanifold.
Recall the computations \cite[Theorem E]{CST4} and \cite[Proposition 7.1.5]{FriedmanBook},
\begin{equation}\label{equa:conesbu}
\crH^j_{\ov{p}}(\rc L;R)=\left\{\begin{array}{ccl}
\crH^j_{\ov{p}}(L;R)&\text{if}&j\leq \ov{p}(\tv),\\
0&\text{if}&j> \ov{p}(\tv),
\end{array}\right.
\end{equation}
\begin{equation}\label{equa:conesgm}
\gH^j_{D\ov{p}}(\rc L;R)=\left\{\begin{array}{ccl}
\gH^j_{D\ov{p}}(L;R)&\text{if}&j\leq \ov{p}(\tv),\\[.1cm]
\Ext(\gH_{j-1}^{D\ov{p}}(L;R),R)&\text{if}&j=\ov{p}(\tv)+1,\\[.1cm]
0&\text{if}&j> \ov{p}(\tv)+1.
\end{array}\right.
\end{equation}
Moreover, we have also (see \cite[Proposition 2.18]{CST2}),
\begin{equation}\label{equa:conesbuc}
\crH^j_{\ov{p},c}(\rc L;R)=\left\{\begin{array}{ccl}
\crH^{j-1}_{\ov{p}}(L;R)&\text{if}&j\geq \ov{p}(\tv)+2,\\
0&\text{if}&j< \ov{p}(\tv)+2.
\end{array}\right.
\end{equation}
From \eqref{equa:conejoin}, 
\eqref{equa:conesgm} and
the exact sequence associated to the pair 
$(\rc L, \rc L \menos \{\tv\})$, we deduce,
\begin{equation}\label{equa:conesgmc}
\gH^j_{D\ov{p},c}(\rc L;R)=\left\{\begin{array}{ccl}
\gH^{j-1}_{D\ov{p}}(L;R)&\text{if}&j\geq \ov{p}(\tv)+3,\\[.1cm]
\Free \gH^{j-1}_{D\ov{p}}(L;R)&\text{if}&j=\ov{p}(\tv)+2,\\[.1cm]
0&\text{if}&j< \ov{p}(\tv)+2.
\end{array}\right.
\end{equation}
From \eqref{equa:conesbuc} and \eqref{equa:conesgmc}, we get,  
\begin{equation}\label{cono}
\crR^j_{\ov{p}}(\rc L;R)=
\left\{\begin{array}{ccl}
\crR^j_{\ov{p}}(L;R)&\text{if}&j\leq \ov{p}(\tv)-1,\\[.1cm]
\coker \left\{\chi_{\ov p} \colon \crH^j_{\ov{p}}(L;R) \to \gH^j_{D\ov{p}}(L;R)\right\}
&\text{if}&j = \ov{p}(\tv),\\[.1cm]
\Tors \gH^{j}_{D\ov{p}}(L;R) &\text{if}&j=\ov{p}(\tv)+1,\\[.1cm]
0&\text{if}&j\geq  \ov{p}(\tv)+2.
\end{array}\right.
\end{equation}

In the particular case of an oriented,  compact \emph{manifold} $M$, we get the peripheral complexes,
\begin{itemize}
\item $\crR_{\ov{p}}^*(\rc M;R)=\crR_{\ov{p}}^{\ov{p}(\tv)+1}(\rc M;R)
=\Ext(H_{\ov{p}(\tv)}(M;R),R)
= \Tors H^{\ov{p}(\tv)+1}(M;R)$,
\item $\crR_{\ov{p},c}^*(\rc M;R)=\crR_{\ov{p},c}^{\ov{p}(\tv)+1}(\rc M;R)
=\Ext(H_{\ov{p}(\tv)}(M;R),R)
= \Tors H^{\ov{p}(\tv)+1}(M;R)$.
\end{itemize}
Let $M$ be $(n-1)$-dimensional.
Observe that $(\ov{p}(\tv)+1)+(D\ov{p}(\tv)+1)=n$.
Thus, the non-singular pairing of the torsion part of the peripheral cohomology (see \thmref{thm:dualityperiph})
corresponds to the classical Poincar\'e duality of the manifold $M$,
$$
\Tors H^{\ov{p}(\tv)+1}(M;R)\otimes \Tors H^{D\ov{p}(\tv)+1}(M;R)\to QR/R.
$$
Moreover, the condition ``locally $(\ov{p},R)$-torsion free'' of \defref{def:locallyfree} for $\rc M$ is exactly 
what we need for having an acyclic peripheral complex and thus a non-singular pairing in blown-up intersection cohomology, since
$$\Tors \gH_{D\ov{p}(\tv)}^{\ov{p}}(M;R)=\Tors H_{n-2-\ov{p}(\tv)}(M;R)
\cong \Tors H_{\ov{p}(\tv)}(M;R) =
\Tors H^{\ov{p}(\tv)+1}(M;R)
=\crR_{\ov{p}}^*(\rc M;R).
$$
\end{example}

\begin{example}[\emph{Isolated singularities}]\label{exam:isolated}
Let $X$ be a stratified pseudomanifold of dimension $n$ with isolated singularities $\Sigma$. 
Let $\R^m\times \rc L^a$ be a conical chart for any singularity $a\in\Sigma$
and set $U=\cup_{a\in \Sigma}\,\R^m\times \rc L^a$. 
As there is no singularity on $V=X\backslash \Sigma$, 
the peripheral and compact peripheral cohomologies of $V$ and $U\cap V$ are reduced to 0. From the
Mayer-Vietoris sequences, we get
$$\crR^*_{\ov{p}}(X;R)=\crR_{\ov{p}}^{\ov{p}(a)+1}(X;R)=
\oplus_{a\in\Sigma} \crR^{\ov{p}(a)+1}_{\ov{p}}( \rc L^a;R)=
\oplus_{a\in\Sigma}\Tors H^{\ov{p}(a)+1}(L^a;R)
$$
\text{ and }
$$\crR^*_{\ov{p},c}(X;R)=
\crR_{\ov{p},c}^{\ov{p}(a)+1}(X;R)=
\oplus_{a\in\Sigma} \crR^{\ov{p}(a)+1}_{\ov{p},c}( \rc L^a;R)=
\oplus_{a\in\Sigma}\Tors H^{\ov{p}(a)+1}(L^a;R).
$$
\end{example}

\begin{example}[\emph{Non-singular torsion free pairing with degenerate torsion pairing}]\label{exam:susprp3cie}
Let $M$ be an oriented compact manifold. From the previous example, we deduce:
$$\crR_{\ov{p}}^*(\Sigma M;R)=
\crR^*_{\ov{p}}(\rc M;R)\oplus \crR^*_{\ov{p}}(\rc M;R)=
\Tors H^{\ov{p}(\tv)+1}(M;R)
\oplus
\Tors H^{\ov{p}(\tv)+1}(M;R).$$
We also have
$$\crT^*_{\ov{p},C}(\Sigma M;R)=\crT_{\ov{p},C}^{\ov{p}(\tv)+1}(\Sigma M;R)=\Tors H^{\ov{p}(\tv)+1}(M;R)
$$
and
$$\crT^*_{\ov{p},K}(\Sigma M;R)=\crT_{\ov{p},K}^{\ov{p}(\tv)+2}(\Sigma M;R)=\Tors H^{\ov{p}(\tv)+1}(M;R).
$$
Thus, here,  the duality of \propref{prop:torsioncompdual} is given by the Poincar\'e duality on the manifold~$M$.
Moreover, as \emph{$\crF^*_{\ov{p}}(X;R)=0$, the torsion free pairing $\Phi_{\ov{p}}$ of (\ref{equa:freepairingR}) is non-singular for any perversity.}
Let us consider two examples where the torsion pairing $L_{\ov{p}}$ of
(\ref{equa:torpairingR}) is degenerate.
\begin{enumerate}[a)]
\item Consider $X=\Sigma \R P^3$,  $R=\Z$ and $\ov{p}=\ov{1}=D\ov{p}$  (the middle perversity), we have
$$\crR^*_{\ov{1}}(X;\Z)=
\crT^3_{\ov{1},K}(X;\Z)\oplus \crT^2_{\ov{1},C}(X;\Z)
=\Z_2\oplus \Z_2.$$
\item Consider $X=\Sigma(S^1\times S^1\times \R P^3)$,  $R=\Z$ and $\ov{p}=\ov{2}=D\ov{p}$  (the middle perversity), 
we have
$$\crR^*_{\ov{2}}(X;\Z)=
\crT^4_{\ov{2},K}(X;\Z)\oplus \crT^3_{\ov{2},C}(X;\Z)
=\Z_2^2\oplus \Z_2^2.$$
Here, in contrast with a), the torsion free pairing is non-trivial.
\end{enumerate}
\end{example}

The next example is an illustration of a peripheral cohomology which comes from the  
torsion free part of the map $\chi_{\ov{p}}$.

\begin{example}[\emph{Singular torsion free pairing with non-singular torsion pairing}]\label{exam:thoms2}
We present two examples of Thom space built from the circle space of a manifold $B$
relatively to an Euler class $e$.
\begin{enumerate}[a)]
\item We choose $B=S^2$, $R=\Z$, $\ov{p}=D\ov{p}=\ov{1}$ and $e=2w$ where $w\in H^2(S^2;\Z)$
is a generator.
This example has been described in \cite[Example 4.10]{CST2} by using the Thom isomorphism and
the Gysin sequence. We deduce from this reference that the map
$\chi_{\ov{1}}\colon \crH^2_{\ov{1}}(X;\Z)=\Z\to \gH^2_{\ov{1}}(X;\Z)=\Z$
is the multiplication by 2. 
Since both cohomologies are abelian  free groups, we have
$$\crR^*_{\ov{1}}(X;\Z)=\crR^2_{\ov{1}}(X;\Z)=\crF^2_{\ov{1}}(X;\Z)=\Z_{2}.$$
Thus, \emph{the torsion pairing $L_{\ov{1}}$ of (\ref{equa:torpairingR}) is non-singular and 
the torsion free pairing $\Phi_{\ov{1}}$ of (\ref{equa:freepairingR}) is singular.} 
\item  We choose $B=\R P^3\times \C P^2\times S^1$, $R=\Z$, $\ov{p}=D\ov{p}=\ov{4}$ and 
$e=(\alpha,3\omega,0)$, where $\alpha\in H^2(\R P^3;\Z)$ and $\omega\in H^2(\C P^2;\Z)$ 
are generators. We have,
$$\crR^*_{\ov{4}}(X;\Z)=\crR^5_{\ov{4}}(X;\Z)=\crF^5_{\ov{4}}(X;\Z)=\Z_{3}\oplus \Z_{3}.
$$
We leave the details to the reader. Here, in contrast with a), we have a non-trivial torsion pairing.
\end{enumerate}
\end{example}

For the suspension of $ \R P^3$, the peripheral cohomology comes entirely from the torsion in
cohomology; in other words, this is a case where the sequence in \cite[Remark 9.2.(3)]{GS} is exact.
This is not the case in the following example.

\begin{example}[\emph{Singular torsion free pairing with degenerate torsion pairing}]
\label{exam:pastrivial}
We consider the Thom space built from the circle bundle over $S^2\times \R P^3\times S^3$,
relatively to the Euler class $e=(3\omega,a,0)$, where
$\omega\in H^2(S^2;\Z)$ and $a\in H^2(\R P^3;\Z)$ are generators.
We choose $R=\Z$, $\ov{p}=D\ov{p}=\ov{4}$.
With the same process than in \examref{exam:thoms2} a), we prove that
$\chi^*_{\ov{4}}\colon \crH_{\ov{4}}^k(X;\Z)\cong \gH_{\ov{4}}^k(X;\Z)$, if $k\neq 5,\,6$, and that
$$\chi^*_{\ov{4}}\colon \crH_{\ov{4}}^5(X;\Z)=\Z\oplus \Z\to 
\gH_{\ov{4}}^5(X;\Z)=\Z\oplus\Z\oplus \Z_{2}$$
is defined by $\chi^*_{\ov{4}}(a,b)=(3a,3b,\ov{b})$. Finally, we have
$\chi^*_{\ov{4}}\colon \crH_{\ov{4}}^6(X;\Z)=\Z_{2}\to \gH_{\ov{4}}^6(X;\Z)=0$.
We compute $$\crR^5_{\ov{4}}(X;\Z)=\Z_{6}\oplus \Z_{6}.$$
 If we go deeper in the torsion  and the torsion free parts of $\chi_{\ov{4}}^*$, 
we get
$$\begin{array}{ccccl}
\crF^*_{\ov{4}}(X;\Z)&=&\crF^5_{\ov{4}}(X;\Z)&=&\Z_{3}\oplus \Z_{3},\\[.1cm]
\crT^*_{\ov{4},K}(X;\Z)&=&\crT^6_{\ov{4},K}(X;\Z)&=&\Z_{2},\\[.1cm]
\crT^*_{\ov{4},C}(X;\Z)&=&\crT^5_{\ov{4},C}(X;\Z)&=&\Z_{2}.
\end{array}$$
Here, \emph{the torsion free pairing $\Phi_{\ov{4}}$ is singular and the torsion pairing $L_{\ov{4}}$ is degenerate. 
Moreover, 
the exact sequence (\ref{equa:cokerexact2}) is non-trivial:}
$0\to\Z_{3}\oplus \Z_{3}\to (\Z_{6}\oplus \Z_{6})/\Z_{2}\to \Z_{2}\to 0$.
\end{example}

In the previous examples, the peripheral cohomology is non-trivial and the 
stratified pseudomanifold is not locally $\ov{p}$-torsion free.  
 The general situation can be more elaborate. 
We first study the peripheral cohomology of the suspension of a stratified homeomorphism 
of a stratified pseudomanifold, see \eqref{equa:susp}. 
Next, we give a specific example of a non-locally $\ov{p}$-torsion free 
stratified pseudomanifold with trivial peripheral cohomology.

\begin{example}{[\emph{Suspension of a stratum-preserving homeomorphism.}]} 
\label{exam:periphsuspensionmap}
Let $(L,\ov{p})$ be a stratified pseudomanifold and $f\colon L\to L$ a stratified homeomorphism, cf.
\cite[Definition 1.5]{CST4}. 
It induces homomorphisms,
$f^*\colon \crH^*_{\ov{p}}(L;R)\to \crH^*_{\ov{p}}(L;R)$
and
$f^*\colon \gH^*_{D\ov{p}}(L;R)\to \gH^*_{D\ov{p}}(L;R)$
(see \cite[Proposition 3.5]{CST4} and \cite[Proposition 3.11]{CST3})
and therefore 
$f^*\colon \crR^*_{\ov{p}}(L;R)\to \crR^*_{\ov{p}}(L;R)$.
The 
\emph{suspension of $f$} is the quotient
\begin{equation}\label{equa:susp}
X=L\times [0,1]/\sim,
\end{equation}
with $(x,0)\sim (f(x),1)$ for any $x\in L$. 
We obtain a stratified pseudomanifold relatively to the filtration
$X_{k}=L_{k}\times [0,1]/\sim$. Locally, this stratified pseudomanifold is stratified homeomorphic to
$L\times J$, where $J\subset \R$ is an interval. So, the perversity $\ov{p}$ on $L$ extends naturally to a perversity
on $X$, also denoted $\ov{p}$.
We cover $X$ with two open subsets, $\{U,V\}$, where 
$$U=(L\times ([0,1[\menos \{1/2\})/\sim 
\text{ and }
V=(L\times ]0,1[)/\sim =L\times ]0,1[.$$
We have $U\cap V=L\times (]0,1[\menos \{1/2\})$ and the restriction map in the Mayer-Vietoris sequence,
$\crR^k_{\ov{p}}(U)\oplus \crR^k_{\ov{p}}(V)\to \crR^k_{\ov{p}}(U\cap V)$,
 becomes
\begin{equation}\label{equa:periphsuspension}
\nu\colon \crR^k_{\ov{p}}(L) \oplus \crR^k_{\ov{p}}(L)\to \crR^k_{\ov{p}}(L) \oplus \crR^k_{\ov{p}}(L),
\end{equation}
with $\nu(x,y)=(x-y,x-f^*(y))$. 
The correspondences $(x,y)\mapsto x$ and $(x,y)\mapsto y-x$ giving isomorphisms, 
$\ker \nu \cong \ker (f-\id)^*$ and $\coker \nu\cong \coker (f-\id)^*$,
 the Mayer-Vietoris sequence reduces to short exact sequences
\begin{equation}\label{equa:MVsuspension}
\xymatrix@1{
0\ar[r]&
\coker(f^*-\id)^k\ar[r]&
\crR^k_{\ov{p}}(X)\ar[r]&
\ker(f^*-\id)^{k+1}\ar[r]&
0.
}
\end{equation}
\end{example} 

\begin{example}[\emph{Pseudomanifold which is not locally $\ov{p}$-torsion free and whose 
$\ov{p}$-intersection homology 
has a Poincar\'e duality}]\label{exam:periphnofree}
With the notation of \examref{exam:periphsuspensionmap}, we choose the stratified pseudomanifold
$L=\Sigma(S^1\times S^1\times \R P^3)$, $R=\Z$ and $\ov{p}=D\ov{p}=\ov{2}$. The corresponding
 peripheral cohomology can be determined from \examref{exam:susprp3cie} as
 \begin{eqnarray}\label{equa:sustorerp3}
 \crR^*_{\ov{2}}(L)=\crR^3_{\ov{2}}(L)&=&
 \Tors H^3(S^1\times S^1\times \R P^3)\oplus  \Tors H^3(S^1\times S^1\times \R P^3) \\
 &=&
 H^1(S^1\times S^1)\otimes H^2(\R P^3)\oplus H^1(S^1\times S^1)\otimes H^2(\R P^3)
 =\Z_2^2\oplus \Z_2^2. \nonumber
 \end{eqnarray}
 Let  $H^1(S^1\times S^1)=\Z[a,b]$,
 $H^2(\R P^3)=\Z_2[u]$
 be described by generators. Then, using the cross product, we have
 \begin{equation}\label{equa:sustorerp3b}
 \crR^3_{\ov{2}}(L)=\Z_2[a\times u]\oplus  \Z_2[b\times u] \oplus \Z_2[a'\times u]\oplus  \Z_2[b'\times u].
  \end{equation}
 For the stratified homeomorphism $f\colon L\to L$, we choose the suspension of the map
 $g\colon S^1\times S^1\times \R P^3\to S^1\times S^1\times \R P^3$, defined by
 $g(x,y,z)=(x+y,-x,z)$. Using the notations of (\ref{equa:sustorerp3}), the induced endomorphism 
 of the peripheral cohomology group $\crR^3_{\ov{2}}(L)$ satisfies
 \begin{eqnarray}\label{equa:labonnevaleur}
  f^*(a\times u)=(a+b)\times u,&& f^*(b\times u)=-a\times u,\\
  f^*(a'\times u)=(a'+b')\times u,&& f^*(b'\times u)=-a'\times u.\nonumber
 \end{eqnarray}
 We can now prove the two required properties on the stratified pseudomanifold $X$ obtained from the suspension of $f$.
 \begin{enumerate}[(i)]
 \item The link of singular points of $X$  is the product
 $ (S^1\times S^1\times \R P^3)$ which verifies
 $$\Tors \gH_{2}^{\ov{2}} (S^1\times S^1\times \R P^3)=\Tors H_{2} (S^1\times S^1\times \R P^3)=\Z_2\oplus \Z_2\neq 0.$$
 Thus $X$ is not $(\ov{2},\Z)$-locally torsion free.
 \item From (\ref{equa:labonnevaleur}) and (\ref{equa:MVsuspension}), we deduce
 $\ker(f^*-\id)=\coker(f^*-\id)=0$ and the triviality of the peripheral cohomology of $X$. 
 Therefore, the $\ov{p}$-intersection homology of $X$ satisfies Poincar\'e duality.
 \end{enumerate}
This example is inspired by an example of H.~King (\cite[\S 1]{MR642001}). The devotees of Riemannian foliations can
 observe a similar idea in an example of Y. Carri\`ere (\cite{MR755161}).
\end{example}

%%%%%%%%%%%%%%%%%%
\begin{example}[\emph{Relative complex of a suspension}]\label{Dpp}

This example shows that the homology of the relative complex of \secref{sec:relativeDp} is not entirely torsion,
in contrast to the peripheral cohomology.
 Let $M=\C P^2\times S^1$ with the perversities $\ov{p} =1$, $D\ov{p}=3$.
Using the Mayer-Vietoris sequence and the classical conical calculation, one gets
  $$
  \crH^{j}_{D\ov p/ \ov p} (\Sigma M;\Z) = 
  \left\{
  \begin{array}{cc}
   \Z \oplus \Z & \hbox{if } j = 2,\,3,\\
   0& \hbox{if not.}
   \end{array}
   \right.
   $$
\end{example}

%%%%%%%%%%%%%%%%%%%%
\providecommand{\bysame}{\leavevmode\hbox to3em{\hrulefill}\thinspace}
\providecommand{\MR}{\relax\ifhmode\unskip\space\fi MR }
\providecommand{\MRhref}[2]{%
  \href{http://www.ams.org/mathscinet-getitem?mr=#1}{#2}
}
\providecommand{\href}[2]{#2}

\end{document}